\documentclass{amsart}

\usepackage{amssymb}
\usepackage{mathtools}

\newtheorem{assumption}[equation]{Assumption}
\newtheorem{corollary}[equation]{Corollary}
\newtheorem{lemma}[equation]{Lemma}
\newtheorem{proposition}[equation]{Proposition}
\newtheorem{remark}[equation]{Remark}
\newtheorem{theorem}[equation]{Theorem}
\numberwithin{equation}{section}
\numberwithin{figure}{section}

\DeclareMathOperator{\ord}{ord}
\DeclareMathOperator{\divisor}{div}
\DeclareMathOperator{\torsion}{torsion}
\DeclareMathOperator{\Pic}{Pic}

\DeclareMathOperator{\im}{Im}
\DeclareMathOperator{\re}{Re}
\DeclarePairedDelimiter\abs{\lvert}{\rvert}

\newcommand{\trisymbol}[6]{T_{#1,#2;#3,#4;#5,#6}}
\newcommand{\recsymbol}[6]{R_{#1,#2,#3;#4,#5,#6}}
\newcommand{\detsymbol}[2]{[#1,#2]}

\def\a{\alpha}
\def\b{\beta}
\def\c{\gamma}
\def\l{\lambda}
\def\v{\varepsilon}

\def\C{\mathbb C}
\def\F{\mathbb F}
\def\M{\mathbb{M}}
\def\N{\mathbb{M}'}
\def\P{\mathbb P}
\def\Q{\mathbb Q}

\def\R{\mathbb R}
\def\Z{\mathbb Z}

\def\CC{\mathcal{C}}
\def\DD{\mathcal{D}}
\def\OO{\mathcal{O}}

\def\PP{\mathfrak{P}}
\def\dd{\textup{d}}
\def\tP #1 #2 {\widetilde P_{#1,#2}}
\def\tPs #1 {\widetilde P_{#1}}

\def\ge{\geqslant}
\def\le{\leqslant}

\def\Spec{\textup{Spec}}

\def\hdr{H_{\textup{dR}}}
\def\INTC{K_2^T(C)_{\textup{int}}}

\def\ol{\overline}

\begin{document}

\title{On $K_2$ of certain families of curves}

\author{Hang Liu}
\address{School of Mathematical Sciences\\University of the Chinese Academy of Sciences\\Beijing 100049\\China}

\author{Rob de Jeu}
\address{Faculteit der Exacte Wetenschappen\\Afdeling Wiskunde\\VU University Amsterdam\\De Boelelaan 1081a\\1081 HV Amsterdam\\The Netherlands}

\begin{abstract}
We construct families of smooth, proper, algebraic curves in characteristic 0,
of arbitrary genus $g$, together with $g$ elements in the kernel of the tame symbol.
We show that those elements are in general independent by a limit
calculation of the regulator.
Working over a number field, we show that in some of those families
the elements are integral.
We determine when those curves are hyperelliptic, finding, in
particular, that over any number field we have non-hyperelliptic curves of all composite genera $ g $
with $ g $ independent integral elements in the kernel of the
tame symbol.
We also give families of elliptic curves over real quadratic
fields with two independent integral elements.
\end{abstract}



\maketitle

\section{Introduction} \label{introduction}

Let $ k $ be a number field, with ring of algebraic integers $ \OO_k $.
The classical relation between the residue at $ s=1 $ of the
zeta-function $ \zeta_k(s) $ and the regulator of $ \OO_k^* \simeq K_1(\OO_k) $,
was generalized by Borel to a relation between a regulator defined
for $ K_{2n-1}(k) = K_{2n-1}(\OO_k) $ and $ \zeta_k(n) $ for
$ n \ge 2 $ \cite{Borel77}.
Inspired by this, Bloch a few years later considered
CM-elliptic curves $ E $ over $ \Q $, and proved a relation between a regulator
associated to an element in $ K_2(E) $ and the value of its $ L $-function
at~2. This was finally published in \cite{Bl78}.
Beilinson in the meantime had made a very general conjecture
about similar relations between regulators of certain $ K $-groups
of regular, projective varieties over number fields, and values
of their $ L $-functions at certain integers (see, e.g., \cite{Schneider}).

Below we shall briefly review Beilinson's conjecture
on $ K_2 $ of a curve over
a number field. For more details  we refer the reader to
the first three sections of \cite{Jeu06}.

Let $ F $ be a field.
By a famous theorem of Matsumoto (see \cite[Theorem 4.3.15]{Rosenberg}),
the group $K_2(F)$ can be described explicitly as
$$F^* \otimes_\Z F^* \slash \langle a\otimes (1-a), a\in F, a\neq 0,1 \rangle,$$
where $\langle \cdots \rangle$ denotes the subgroup generated by the indicated elements.
The class of $a \otimes b$ is denoted $\{a,b\}$, so that $K_2(F)$ is an Abelian group
(written additively), with generators $\{a,b\}$ for $a$ and $b$ in $F^*$, and relations
\begin{eqnarray*}
\{a_1a_2,b\} &=& \{a_1,b\} + \{a_2,b\}\\
\{a,b_1b_2\} &=& \{a,b_1\} + \{a,b_2\}\\
\{a,1-a\} &=& 0 \quad \text{if } a \text{ is in } F, a\neq 0,1.
\end{eqnarray*}
These relations also imply $\{a, -a\}=0$ and $\{a, b\}=-\{b, a\}$.

Suppose that $ F $ is the function field of a regular, proper,
irreducible curve $ C $ over a field $ k $.
Then we let
\begin{equation*}
K_2^T(C) = \ker \biggl( K_2(F) \overset{T}{\longrightarrow} \bigoplus_{x\in C^{(1)}} k(x)^* \biggr)
\,,
\end{equation*}
where $ C^{(1)} $ denotes the set of closed (codimension~1) points
of $ C $, and the $x$-component of the map $T$ is the \emph{tame symbol} at $x$, defined
on generators by
\begin{equation}\label{eqn:tame}
T_x:\{a,b\} \mapsto (-1)^{\ord_x(a)\ord_x(b)}\frac{a^{\ord_x(b)}}{b^{\ord_x(a)}}(x).
\end{equation}
For
$\a$ in $K_2(F)$ we have the \emph{product formula}
\cite[Theorem~8.2]{Bass},
\begin{equation} \label{productformula}
\prod_{x\in C^{(1)}} \textup{Nm}_{k(x)/k}(T_x(\a)) = 1.
\end{equation}

Now assume that the base field $ k $ is a number field.  Then
Beilinson's conjecture in its original statement
(see \cite{Bei80})
applied to the group $ K_2(C) \otimes_\Z \Q = K_2^T(C) \otimes_\Z \Q $, but in \cite{Jeu06}
a (slightly stronger) formulation was given without tensoring with $ \Q $, and we
shall use this approach here.  Suppose that $ C $ has genus $ g $,
and that $ k $ is algebraically closed in $ F $ (so that $ C $ is
geometrically irreducible over $ k $).
Then Beilinson originally expected $ K_2^T(C) \otimes_\Z \Q $ to have dimension $ r = g \cdot [k:\Q] $.
However, computer calculations \cite{BG}
showed that this dimension could be larger and that an additional
condition should be used, which led to a modification of the
conjecture \cite{Be85}.
For this, let us fix a regular,
proper model $ \CC/\OO_k $  of $ C/k $, with $ \OO_k $ the ring
of algebraic integers in~$ k $. Then we define
\begin{equation*}
 K_2^T(\CC) = \ker\biggl( K_2(F) \overset{T_\CC}{\rightarrow} \oplus_\DD \F(\DD)^* \biggr)
\,,
\end{equation*}
where $\DD$ runs through all irreducible curves on $ \CC $,
and $\F(\DD)$ is the residue field at $\DD$.
The component of $T_\CC$
for $ \DD $ is given by the tame symbol corresponding to $\DD$ similar to~\eqref{eqn:tame},
\begin{equation*}
\{a,b\} \mapsto
(-1)^{v_{\DD}(a)v_{\DD}(b)}\frac{a^{v_{\DD}(b)}}{b^{v_{\DD}(a)}}(\DD)
\,,
\end{equation*}
where $v_{\DD}$ is the valuation on $F$ corresponding to $\DD$.
Because the $ \DD $ that surject onto $ \Spec(\OO_k) $ correspond
exactly to the points $ x $ in $ C^{(1)} $ under localization, and the
formula for the $ \DD $-component of $ T_\CC $ localizes to the
one in~\eqref{eqn:tame} for the corresponding point~$ x $, it
follows that $ K_2^T(\CC) $ is a subgroup of $ K_2^T(C) $.
It was stated without proof on \cite[p.~344]{Jeu06} that the
image of $ K_2^T(\CC) $ in $ K_2^T(C) $ modulo torsion is
independent of the choice of $ \CC $, but in fact the subgroup
$ K_2^T(\CC) $ of $ K_2^T(C) $ is independent of this choice
(see Proposition~\ref{modelchoice} below).
We shall denote it by $ \INTC $, and call its elements \emph{integral}.

Let $ X $ be the complex manifold associated to $ C \times_\Q \C $.
It is a disjoint union of $ [k:\Q] $ Riemann surfaces of genus~$ g $,
and complex conjugation acts on it through the action on $ \C $
in $ C \times_\Q \C $.  We let
$ \hdr^1(X , \R)^- $ consist of those elements that are multiplied
by $ -1 $ under the resulting action on $ X $.
It is a real vector space of dimension~$ r $.
It can be paired with $ H_1(X, \Z)^- $, the part
of $ H_1(X, \Z) $ on which complex conjugation on $ X $ induces
multiplication by $ -1 $, which is isomorphic with~$ \Z^r $.
Dividing Beilinson's regulator map by $ i $ we obtain a map
$ K_2^T(C) /\torsion  \to \hdr^1(X, \R)^- $,
and combining this with the pairing we obtain
the \emph{regulator pairing}
\begin{equation} \label{eqn:pairing}
\begin{split}
\langle \,\cdot\, , \,\cdot\, \rangle : H_1(X;\Z)^- \times K_2^T(C)/\torsion \to \R
\\
(\c, \a) \mapsto
\frac{1}{2\pi}\int_{\c}\eta(\a)
\,,
\end{split}
\end{equation}
with $ \eta(\a) $ obtained by writing $ \a $ as a sum of symbols
$ \{a,b\} $, and mapping $ \{a,b\} $ to
\begin{equation} \label{eqn:eta}
\eta(a,b) = \log|a| \dd \arg(b) - \log|b| \dd \arg(a)
\,,
\end{equation}
and $\c$ is
chosen such that $\eta(\a)$ is defined. The pairing is well-defined
\cite[Section~3]{Jeu06}.  If $\c_1,\cdots,\c_r$ form a
basis of $H_1(X;\Z)^-$, and $M_1,\cdots,M_r$ are in
$ K_2^T(C) $ or $K_2^T(C)/\torsion$, we can define the regulator $R(M_1,\cdots,M_r)$ by
\begin{equation}\label{eqn:regulator}
R=|\det(\langle \c_i, M_j \rangle)|.
\end{equation}
Beilinson expects $ \INTC \otimes_\Z \Q $ to have $ \Q $-dimension
$ r $, that $ R \ne 0 $ if $ M_1,\dots, M_r $ form a basis of
it, and that $ R $ is related to the value of $ L(H^1(C),s) $ at $ s=2 $
(see \cite[Conjecture~3.11]{Jeu06}).

The proof of Proposition~\ref{modelchoice} shows that
$ \INTC $ is a quotient of $ K_2(\CC) $, which is expected to
be finitely generated by a conjecture of Bass. If that is the
case, and $ M_1,\dots,M_r $ form a $ \Z $-basis of $ \INTC/\torsion $,
then $ R $ is also independent of the choice of this basis.

That $ \INTC $ is finitely generated is only known if $ g=0 $.
Most of the work on the conjecture has been put into constructing $ r $ independent
elements in $ \INTC/\torsion $ and, if possible, relating the
resulting regulator with the $ L $-value either numerically
or theoretically (see, e.g., \cite{De89, De90, Ross92, Kimura, rol, scha-rols98, Jeu06,Otsubo}).

The goal of this paper is
twofold. Firstly, we construct $ g $ elements in $K_2^T(C)$ on certain families of curves of genus $ g $.
The curves here are in general not hyperelliptic but include
the curves in \cite{Jeu06} as a special case.
If the base field is a number field then we show that, for suitable
parameters, the elements are integral.
Secondly, we consider the curves in certain families over a 1-dimensional
parameter space as a fibred surface, construct a family of loops
in the fibres, and establish a limit formula of the regulator
pairing as a
function of the parameter.  This in the end amounts to a residue
calculation on the bad fibre.
As a consequence we find that, over any number field,
we obtain families with $ g $ linearly independent elements
in $ \INTC $.
In particular, if the base field is $ \Q $ then we have as many
linearly independent elements in $ \INTC $ as predicted by Beilinson's
conjecture.

The authors wish to thank Ulf K\"uhn and
Steffen M\"uller for discussing their
constructions using hyperflexes in \cite{K-S} in preliminary form.
Those played an important role in the process of arriving at our, in the end entirely
different, construction in Section~\ref{section:construction}.
The authors would like to thank Xi Chen, James Lewis,
Deepam Patel and Tomohide Terasoma for useful conversations and/or correspondence,
as well as the referees for useful comments.

Hang Liu gratefully acknowledges financial support by
the China Scholarship Council and the National Natural Science Foundation of China (No. 11371343),
as well as the hospitality of VU University Amsterdam.

The structure of the paper is as follows.
In Section~\ref{section:construction}, we construct elements of $K_2^T$ on some families of curves.
The equations of the curves and the elements in $ K_2^T(C) $
are all based on lines in $ \P^2 $; see~\eqref{eqn:curve}, \eqref{eqn:rectangle}
and~\eqref{eqn:triangle}.
By giving explicit relations between these elements in
Lemma~\ref{relations},
we show that in general the number of independent elements
we find is at most the genus of the curve (see Proposition~\ref{prop:basis}
and Remark~\ref{remark:genus}).
In Section~\ref{section:integrality},
we show in Proposition~\ref{modelchoice} that the definition
of $ \INTC = K_2^T(\CC) $ is independent of the model $ \CC $, and in
Theorem~\ref{thm:main} that, for some of our families, the
elements we constructed are in $ \INTC $.
Proposition~\ref{prop:hyperelliptic} makes precise when those curves
are hyperelliptic, and shows in particular that they are not
in general. In the remainder of Section~\ref{section:hyperelliptic}
we consider our hyperelliptic families that are not covered by~\cite{Jeu06}.
In Section~\ref{section:linearindependence}, we use $ \C $ as
our base field and in Theorem~\ref{thm:limit}
prove a limit formula for the regulator pairing under some conditions, which shows that the elements
that we constructed are, `in general', linearly independent.
In particular, in Corollary~\ref{cor:int} we find that the families of curves
of genus $ g $ that we
constructed in Section~\ref{section:integrality}, have, in general, $ g $
independent elements in $ \INTC $.
We finish that section by also giving examples of elliptic curves over real quadratic
fields with two independent integral elements.

\section{Construction of elements in $K_2^T$ for certain curves} \label{section:construction}

Let $ C $ be a regular, proper, irreducible curve over a field, and let $ F $
be its function field.
The simplest way to construct elements in $K_2^T(C)$ is to use two
functions with only three zeroes and poles in total, as we now
recall from~\cite[Construction 4.1]{Jeu06}.

Assume $P_1,P_2,P_3$ are distinct rational points of $ C $ whose pairwise differences are
torsion divisors.  Thus, using indices modulo~3, there are rational functions $f_i$ with
$ \divisor(f_i)=m_i(P_{i+1})-m_i(P_{i-1}) $, where $m_i$ is the order of
$(P_{i+1})-(P_{i-1})$ in the divisor group $\Pic^{0}(C)$. We then
define three elements of $K_2(F)$ by
\begin{equation*}
\Big \{ \frac{f_{i+1}}{f_{i+1}(P_{i+1})}, \frac{f_{i-1}}{f_{i-1}(P_{i-1})} \Big \}
\,.
\end{equation*}
Using the product formula~\eqref{productformula},
one sees that those elements are in $K_2^T(C)$.

We shall now construct regular, proper, irreducible curves to
which the above applies. Fix an integer $ N \ge 3 $.
For $ i=1,\dots,N $, let $ L_i $ be a non-constant polynomial of
the form $ a_ix+b_iy+c_i $ such that the lines defined by $ L_i = 0 $
are distinct and pairwise non-parallel.
The affine curve defined by $\prod_{i=1}^N L_i - 1 = 0 $
is irreducible by Lemma~\ref{irrlemma} below.
Its projective closure $ C' $ in $ \P^2 $
has $N$ distinct points $P_i=[-b_i,a_i,0]$ at infinity, all non-singular.
The normalisation $ C $ of $ C' $
contains points $ \tPs i $ corresponding to $ P_i $
for $ i=1,\dots,N $.
Viewing the $ L_i $ as non-zero rational functions on $ C $, we have
$\divisor(\frac{L_i}{L_j}) = N(\tPs i ) - N(\tPs j )$.
So, for every three distinct points $ \tPs {i_1} , \tPs {i_2} . \tPs {i_3} $
among $ \tPs 1 ,\dots, \tPs N $, we obtain elements in $ K_2^T(C) $
using the functions $ f_j = L_{i_k}/L_{i_l} $, where $ \{j,k,l\} = \{1,2,3\} $.

The $P_i$ are hyperflexes of~$ C' $
(in fact, this example was inspired by \cite{K-S}, where hyperflexes are
used in order to construct elements in $K_2^T(C)$ for suitable~$ C $).
However, in order to prove the resulting elements in $ K_2^T(C) $
are integral if our base field is a number field and all our
coefficients are in the ring of integers,
it is crucial that, modulo each prime ideal, all $ P_i $ have different reductions
(see the proof of Theorem~\ref{thm:main}).
This bounds $ N $, and hence the genus of~$ C $.
For example, if the base field is $ \Q $,
this imposes $ N \le 3 $, hence $ N=3 $ and $ g = 1 $.
But it turns out that it is not crucial that the $ P_i $ are
hyperflexes, and that we can allow parallel lines.

Namely, we shall construct regular, proper, irreducible curves
$ C $ together with
elements in $ K_2^T(C) $ that are in general different from the
ones in the set-up at the beginning of this section because the divisors
of the functions involved are more complicated (see Remark~\ref{torsionremark}
below).

For this, let $ N \ge 2 $, and for $ i=1,\dots, N $, let $ N_i \ge1 $.
For $ i=1,\dots,N$ and $ j=1,\dots,N_i $, let $ L_{i,j} $
be a non-constant polynomial $a_ix+b_iy+c_{i,j}$ such that
the lines defined by $ L_{i,j} = 0 $
are distinct, and non-parallel for distinct $ i $.
Consider the affine curve defined by
$ f(x,y) = 0 $ for
\begin{equation} \label{eqn:curve}
f(x,y)=\l \prod_{i=1}^N \prod_{j=1}^{N_i} L_{i,j}-1, \quad \l \neq
0
\end{equation}
with $ \l $ in the base field.

\begin{lemma} \label{irrlemma}
The affine curve defined by $ f(x,y) = 0 $ with $ f(x,y) $ as
in~\eqref{eqn:curve} is irreducible.
\end{lemma}

\begin{proof}
It suffices to show that $ f(x,y) $ is irreducible.
Without loss of generality,
suppose $L_{1,1}=x$ and $L_{2,1}=y$. If $f=(h_1+1)(h_2-1)$ with
$h_1(0,0)=h_2(0,0)=0$,
it is easy to see that $x$ and $y$ both divide $h_1$ and $h_2$.
Furthermore, one can show that all $L_{i,j}$ divide $h_1$ and $h_2$.
Namely, we have $(h_1+1)(h_2-1) \equiv f \equiv -1 $ modulo~$ L_{i,j} $.
If $i=1$, $h_1$ and $h_2$ are equivalent to functions of $y$ without constant term modulo $L_{i,j}$,
which forces $h_1$ and $h_2$ to be $0$ modulo $L_{i,j}$.
If $i\ne 1$, $h_1$ and $h_2$ are equivalent to functions of $x$ without constant term modulo $L_{i,j}$,
which also forces $h_1$ and $h_2$ to be $0$ modulo $L_{i,j}$.
Because the $L_{i,j}$ are pairwise coprime,
$\prod_{i=1}^N \prod_{j=1}^{N_i} L_{i,j}$ divides $h_1$ and $h_2$,
hence $h_1$ or $h_2$ equals 0, showing that $f$ is irreducible.
\end{proof}

Let $ C' $ be the projective closure in $ \P^2 $ of the affine curve defined by $ f(x,y) = 0 $
with $ f(x,y) $ as in~\eqref{eqn:curve}, so $ C' $ has the points $ P_i = [-b_i,a_i,0] $
at infinity.
Let $C$ be its normalisation, with function field~$ F $.
The $ L_{i,j} $ define non-zero rational functions on $ C $
so we view them as elements of $ F^* $.
Note that for each $ i $, the points on $ C $ lying
above $ P_i $ correspond to the tangent lines to~$ C' $ at $ P_i $.
Those correspond to the affine lines defined by $ L_{i,j} = 0 $
for $ j=1,\dots,N_i $. Therefore the points in $ C $ above $ P_i $
are $ \tP i j $ for $ j=1,\dots,N_i $, with $ \tP i j $
corresponding to~$ L_{i,j} $.

We now construct the elements in $ K_2^T(C) $.
In order to simplify the notation, we denote
$a_ib_k-a_kb_i$ by $\detsymbol{i}{k}$, so
$\detsymbol{i}{k} = -\detsymbol{k}{i}$,
and $\detsymbol{i}{k} \neq 0$ if $i \neq k$. We then define two types of elements in $K_2(F)$ by
\begin{eqnarray}
\recsymbol{i}{j}{k}{l}{m}{n} &=& \left\{ \frac{L_{i,j}}{L_{i,k}}, \frac{L_{l,m}}{L_{l,n}}\right\}, \quad i\neq l
\label{eqn:rectangle}
\\
\trisymbol{i}{j}{k}{l}{m}{n} &=& \left\{ \frac{\detsymbol{i}{m}}{\detsymbol{k}{m}} \frac{L_{k,l}}{L_{i,j}}, \frac{\detsymbol{i}{k}}{\detsymbol{m}{k}} \frac{L_{m,n}}{L_{i,j}}\right\}, \quad i,k,m \text{ distinct}
\label{eqn:triangle}
\,.
\end{eqnarray}
The $ R $-element is constructed from two pairs of parallel lines
(forming a parallelogram or rectangle), the $ T $-element from
three pairwise non-parallel lines (forming a triangle).

Clearly
$
 \trisymbol{i}{j}{k}{l}{m}{n}
=
\left\{
\frac{\detsymbol{i}{m} L_{k,l}}{\detsymbol{k}{m} L_{i,j}}, \frac{\detsymbol{i}{k} L_{m,n}}{\detsymbol{i}{m} L_{k,l}}
\right\}
=
 - \trisymbol{k}{l}{i}{j}{m}{n}
$
is alternating under permutation of the three pairs
$ (i,j) $, $ (k,l) $ and $ (m,n) $.
The $ R $-element satisfies similar symmetries. Those will
be stated in Lemma~\ref{relations} below, together with
some relations among the elements, but we first show all elements
have trivial tame symbol.

\begin{lemma} \label{K2Tlemma}
The $\recsymbol{i}{j}{k}{l}{m}{n}$ and $\trisymbol{i}{j}{k}{l}{m}{n}$ are
in $K_2^T(C)$.
\end{lemma}

\begin{proof}
First we compute the tame symbol of $\recsymbol{i}{j}{k}{l}{m}{n}$.
Obviously, $\frac{L_{i,j}}{L_{i,k}}$ and $\frac{L_{l,m}}{L_{l,n}}$
only have zeroes and poles at the $ \tP i * $ and $ \tP l * $,
hence the tame symbol is trivial at all other points.
Since $\frac{L_{i,j}}{L_{i,k}}(\tP l * )=\frac{L_{l,m}}{L_{l,n}}(\tP i * )=1$,
it is also trivial at these points.

Now consider $\trisymbol{i}{j}{k}{l}{m}{n}$.
Denote $\frac{\detsymbol{i}{m}}{\detsymbol{k}{m}} \frac{L_{k,l}}{L_{i,j}}$ and
$\frac{\detsymbol{i}{k}}{\detsymbol{m}{k}} \frac{L_{m,n}}{L_{i,j}}$ by $h_1$ and $h_2$ respectively.
Then $h_1$ has zeroes and poles only at the $ \tP i * $
and $ \tP k * $, and  $h_2$ only at the
$ \tP i * $ and $ \tP m * $.
The tame symbol is trivial at the
$ \tP k * $ and $ \tP m * $ since $h_1(\tP m * )=h_2(\tP k * )=1$.
As the $ T $-element is alternating, the same holds at the~$ \tP i * $.
\end{proof}

\begin{remark} \label{torsionremark}
Although our construction of elements in $ K_2^T(C) $ above
is not based on this, it turns out that all $ (\tP i j ) - (\tP k l ) $ are torsion
divisors.
Namely, if $ d = \deg(f) = \sum_{i=1}^N N_i $,
then
$  \divisor \left( \frac{L_{i_1,j_1}}{L_{i_2,j_2}} \right) $
equals
\begin{equation*}
 \left[d (\tP i_1 j_1 ) + \sum_{l=1}^{N_{i_1}} ( (\tP i_1 l ) - (\tP i_1 j_1 ) ) \right]
-
 \left[d (\tP i_2 j_2 ) + \sum_{l=1}^{N_{i_2}} ( (\tP i_2 l ) - (\tP i_2 j_2 ) ) \right]
\,.
\end{equation*}
Taking $ i_1 = i_2 = i $ shows that $ (\tP i j_1 ) - ( \tP i j_2 ) $
is torsion, and the result is clear.
\end{remark}

\begin{lemma} \label{lemma:smooth}
If the base field has characteristic zero, then for fixed $L_{i,j}$
the affine curve defined by
$ f(x,y) = 0 $ with $ f(x,y) $ as in~\eqref{eqn:curve}
is non-singular except
for finitely many values of $ \lambda $.
\end{lemma}

\begin{proof}
In order to see this, it suffices to show that $\partial_x f$ and $\partial_y f$ are coprime, so that
only finitely many points satisfy $\partial_x f=\partial_y f=0$.
Because those points are independent of $ \l $, this excludes
only finitely many values of~$ \l $.

In order to prove that $\partial_x f$ and $\partial_y f$ are coprime,
suppose an irreducible $ h $ divides $ \partial_x f$
and $\partial_y f$.
The curve defined by $ h=0 $ can meet the
lines defined by $ L_{i,j} = 0 $ only at their
points of intersection, since those are the singularities
of the curve defined by $ \prod_{i=1}^N \prod_{j=1}^{N_i} L_{i,j} = 0$.

Suppose $h$ does not vanish at any of these points.  Then $ h $ restricted
to the lines defined by $ L_{i,j} = 0 $ has no zeroes, hence is a constant. The
union of those lines is connected as $ N \ge 2 $, hence this
is the same constant $ c $ for all lines.
As the $ L_{i,j} $ are coprime, they all divide $ h - c $, which
is impossible because of degrees.

If $ h $ vanishes at one of these points, then we may, without loss of generality,
suppose we have $L_{1,1} = x$ and $L_{2,1} = y$, and that the
point is $ (0,0) $.
Let $u$ be the product of those $ L_{i,j} $ that vanish at $ (0,0)$,
and let $\tilde{u}$ be the product of the other $ L_{i,j} $.
Then $ h $ divides $\tilde{u}\partial_x u + u\partial_x \tilde{u}$ as
well as $\tilde{u}\partial_y u + u\partial_y \tilde{u}$.
By the definition of $u$, it is a homogeneous polynomial,
hence $x\partial_x u + y\partial_y u = \deg(u)u$.
Therefore $h$ divides $ u(\deg(u)\tilde{u}+ x\partial_x \tilde{u} +y \partial_x \tilde{u})$.
Because $\deg(u)\tilde{u}+ x\partial_x \tilde{u} +y \partial_x \tilde{u} \neq 0$ at $(0,0)$ by the definition of $\tilde{u}$,
we find $ h $ divides $u$. But then $ h $ vanishes on one of
the lines, which is impossible.
\end{proof}

The completion $C' $ in $ \P^2 $ of the affine curve defined by
$ f(x,y) = 0 $ with $ f(x,y) $ as in~\eqref{eqn:curve},
has points $ P_i $ ($ i=1,\dots,N $) at infinity, with $ P_i $ of multiplicity $ N_i $.
If $ N_i \ge 2 $ then $ P_i $ is a simple singular point.
Suppose the affine part of the curve is smooth over the base field,
which in characteristic zero is in general the case by Lemma~\ref{lemma:smooth} above.
By the degree-genus formula, the genus $g$ of $C$ then equals
\begin{equation} \label{eqn:genus}
   { \sum_{i=1}^N N_i -1 \choose 2} - \sum_{i=1}^N {N_i \choose 2 }
=
   \sum_{1\le i < j \le N} N_i N_j - \sum_{1\le i \le N} N_i + 1
\,.
\end{equation}
In Section~\ref{section:relations} we shall show that~\eqref{eqn:rectangle} and~\eqref{eqn:triangle}
give us at most this number of linearly independent elements.
So if we take $ \Q $ as our base field, and we can show that these
elements are integral by imposing some condition on the equation of
the lines, and linearly independent,
then we have as many elements as predicted by Beilinson's conjecture. This is
what we shall do in Sections~\ref{section:integrality} and~\ref{section:linearindependence},
but in greater generality.
To complement this, we also show in Section~\ref{section:relations}
that we cannot really get more independent elements out the points~$ \tP i * $,
and in Section~\ref{section:hyperelliptic} we determine for $ N=2 $
and~3 which curves are hyperelliptic and compare those with the
curves studied in~\cite{Jeu06}.

\section{Relations among the elements} \label{section:relations}

In this section we give relations among the elements of type
$ R $ and $ T $ that we constructed
in Section~\ref{section:construction}, and use those
in Proposition~\ref{prop:basis}
to reduce the number of generators for the subgroup $ V $ of
$ K_2^T(C) $ that they generate.

The relations in Lemma~\ref{relations} below
are based on divisions and combinations of polygons
formed by lines
(see Figures~\ref{fig:figure1} and~\ref{fig:figure2}).
For example, if we consider two parallel lines $L_{1,1}$, $L_{1,2}$ and
three parallel lines $L_{2,1}$, $L_{2,2}$, $L_{2,3}$, then the two oriented
parallelograms
$[L_{1,1}, L_{2,1}, L_{1,2}, L_{2,2}]$ and $[L_{1,1}, L_{2,2}, L_{1,2}, L_{2,3}]$,
specified by their sides,
``formally add up'' to $[L_{1,1}, L_{2,1}, L_{1,2}, L_{2,3}]$.
This corresponds to a relation
$ \recsymbol{i}{j}{k}{l}{m}{n}=\recsymbol{i}{p}{k}{l}{m}{n}+\recsymbol{i}{j}{p}{l}{m}{n}$,
which leads to relation~(iii). Similarly, for
two horizontal lines $L_{1,1}$, $L_{1,2}$, two vertical lines $L_{2,1}$, $L_{2,2}$, and one diagonal
line $L_{3,1}$ in general position, one obtains a relation between
four oriented triangles and one oriented rectangle by considering
``formal overlaps and cancellations'', namely
\begin{alignat*}{1}
& 
  [L_{1,1}, L_{2,1}, L_{3,1}] - [L_{1,1}, L_{2,2}, L_{3,1}]
+ [L_{1,2}, L_{2,2}, L_{3,1}] - [L_{1,2}, L_{2,1}, L_{3,1}]
\\
= \, &
 [L_{1,1}, L_{2,1}, L_{1,2},L_{2,2}]
.
\end{alignat*}
This corresponds to relations as in~(iv) and~(v).

\begin{figure}[ht]
\hspace{-8mm}
\begin{minipage}[b]{0.45\linewidth}
\centering

\begin{picture}(100,100)
\put(-20, 0){$L_{1,1}$}
\put(0, 0){\line(1,0){120}}
\put(-20, 50){$L_{1,2}$}
\put(0,50){\line(1,0){120}}
\put(5, 65){$L_{2,1}$}
\put(10, -10){\line(0, 1){70}}
\put(55, 65){$L_{2,2}$}
\put(60, -10){\line(0, 1){70}}
\put(105, 65){$L_{2,3}$}
\put(110, -10){\line(0, 1){70}}
\end{picture}

\caption{}
\label{fig:figure1}
\end{minipage}
\hspace{-0.8cm}
\begin{minipage}[b]{0.45\linewidth}
\centering

\begin{picture}(100,100)
\put(0, 0){$L_{1,1}$}
\put(20, 0){\line(1,0){140}}
\put(0, 30){$L_{1,2}$}
\put(20, 30){\line(1,0){100}}
\put(25, 105){$L_{2,1}$}
\put(30, -10){\line(0,1){110}}
\put(65, 75){$L_{2,2}$}
\put(70, -10){\line(0,1){80}}
\put(0, 90){$L_{3,1}$}
\put(20,90){\line(4,-3){130}}
\end{picture}

\caption{}
\label{fig:figure2}
\end{minipage}
\end{figure}

\begin{lemma} \label{relations}
We have the following relations among the
$\recsymbol{i}{j}{k}{l}{m}{n}$ and $\trisymbol{i}{j}{k}{l}{m}{n}$.
\begin{enumerate}
\item[(i)]
$ \trisymbol{i}{j}{k}{l}{m}{n}=-\trisymbol{k}{l}{i}{j}{m}{n}=-\trisymbol{i}{j}{m}{n}{k}{l} $;

\item[(ii)]
$ \recsymbol{i}{j}{k}{l}{m}{n}=-\recsymbol{i}{k}{j}{l}{m}{n}=-\recsymbol{l}{m}{n}{i}{j}{k} $;

\item[(iii)]
$ \recsymbol{i}{j}{k}{l}{m}{n}=\recsymbol{i}{p}{j}{l}{q}{m}-\recsymbol{i}{p}{k}{l}{q}{m}-\recsymbol{i}{p}{j}{l}{q}{n}+\recsymbol{i}{p}{k}{l}{q}{n} $;

\item[(iv)]
$ \recsymbol{i}{j}{k}{l}{m}{n}=\trisymbol{p}{q}{i}{j}{l}{m}-\trisymbol{p}{q}{i}{k}{l}{m}-\trisymbol{p}{q}{i}{j}{l}{n}+\trisymbol{p}{q}{i}{k}{l}{n} $;

\item[(v)]
$ \trisymbol{i}{j}{k}{l}{m}{n}=\trisymbol{i}{p}{k}{l}{m}{n}-\trisymbol{i}{p}{k}{q}{m}{n}+\trisymbol{i}{j}{k}{q}{m}{n}+\recsymbol{i}{p}{j}{k}{q}{l} $;

\item[(vi)]
$ \trisymbol{i}{j}{k}{l}{m}{n}=\trisymbol{i}{p}{k}{l}{m}{n}-\trisymbol{i}{p}{q}{r}{m}{n}+\trisymbol{i}{j}{q}{r}{m}{n}+\trisymbol{i}{p}{q}{r}{k}{l}-\trisymbol{i}{j}{q}{r}{k}{l} $;

\item[(vii)]
$ \trisymbol{i}{j}{k}{l}{m}{n}=\trisymbol{p}{q}{i}{j}{k}{l}+\trisymbol{p}{q}{k}{l}{m}{n}-\trisymbol{p}{q}{i}{j}{m}{n} $.
\end{enumerate}
\end{lemma}

\begin{proof}
The first two parts are easy consequences of $ \{a, -a\} = 0 $,
$\{a^{-1},b\} = \{a,b^{-1}\} =  -\{a,b\}$
and $\{a,b\}=-\{b,a\}$.
The third follows by working out $ \{a_1 a_2 , b_1 b_2 \} $
into the $ \{a_i,b_j\} $.
Also, (v) is a consequence of~(iv) and~(i)
by taking $i,j,k,l,m,n,p,q$ in (iv) to be $i,p,j,k,q,l,m,n$ respectively,
and~(vi) is a consequence of~(vii) and~(i).
Thus it suffices to prove~(iv) and~(vii).
Here we prove~(vii), as (iv) is quite straightforward.

Note that
$ \trisymbol{i}{j}{k}{l}{m}{n}+\trisymbol{p}{q}{i}{j}{m}{n} = \trisymbol{i}{j}{k}{l}{m}{n} - \trisymbol{i}{j}{p}{q}{m}{n} $
equals
\begin{eqnarray*}
&& \left\{ \frac{\detsymbol{i}{m}}{\detsymbol{k}{m}} \frac{L_{k,l}}{L_{i,j}}, \frac{\detsymbol{i}{k}}{\detsymbol{m}{k}} \frac{L_{m,n}}{L_{i,j}}\right\} - \left\{ \frac{\detsymbol{i}{m}}{\detsymbol{p}{m}} \frac{L_{p,q}}{L_{i,j}}, \frac{\detsymbol{i}{p}}{\detsymbol{m}{p}}\frac{L_{m,n}}{L_{i,j}}\right\}\\
&=&
  \left\{ \frac{\detsymbol{i}{m}}{\detsymbol{k}{m}} \frac{L_{k,l}}{L_{i,j}}, \frac{\detsymbol{i}{k}}{\detsymbol{m}{k}} \frac{L_{m,n}}{L_{i,j}}\right\}
- \left\{ \frac{\detsymbol{i}{m}}{\detsymbol{p}{m}} \frac{L_{p,q}}{L_{i,j}}, \frac{\detsymbol{i}{k}}{\detsymbol{m}{k}} \frac{L_{m,n}}{L_{i,j}}\right\}
\\
&& + \left\{ \frac{\detsymbol{i}{m}}{\detsymbol{p}{m}} \frac{L_{p,q}}{L_{i,j}}, \frac{\detsymbol{i}{k}\detsymbol{m}{p}}{\detsymbol{m}{k}\detsymbol{i}{p}}\right\}\\
&=& \left\{ \frac{\detsymbol{p}{m}}{\detsymbol{k}{m}} \frac{L_{k,l}}{L_{p,q}}, \frac{\detsymbol{i}{k}}{\detsymbol{m}{k}} \frac{L_{m,n}}{L_{i,j}}\right\}
 + \left\{ \frac{\detsymbol{i}{m}}{\detsymbol{p}{m}} \frac{L_{p,q}}{L_{i,j}}, \frac{\detsymbol{i}{k}\detsymbol{m}{p}}{\detsymbol{m}{k}\detsymbol{i}{p}}\right\}.
\end{eqnarray*}
Then
$ \trisymbol{p}{q}{i}{j}{k}{l}+\trisymbol{p}{q}{k}{l}{m}{n} = \trisymbol{k}{l}{m}{n}{p}{q}+\trisymbol{i}{j}{k}{l}{p}{q} $
is equal to
\begin{equation*}
 \left\{ \frac{\detsymbol{p}{i}}{\detsymbol{k}{i}} \frac{L_{k,l}}{L_{p,q}}, \frac{\detsymbol{i}{k}}{\detsymbol{m}{k}} \frac{L_{m,n}}{L_{i,j}}\right\}
 + \left\{ \frac{\detsymbol{m}{k}}{\detsymbol{p}{k}} \frac{L_{p,q}}{L_{m,n}}, \frac{\detsymbol{i}{k}\detsymbol{m}{p}}{\detsymbol{m}{k}\detsymbol{i}{p}}\right\}.
\end{equation*}
Therefore
$ \trisymbol{i}{j}{k}{l}{m}{n}+\trisymbol{p}{q}{i}{j}{m}{n}-(\trisymbol{p}{q}{i}{j}{k}{l}+\trisymbol{p}{q}{k}{l}{m}{n}) $
equals
\begin{eqnarray*}
&&
\left\{ \frac{\detsymbol{p}{m}\detsymbol{k}{i}}{\detsymbol{k}{m}\detsymbol{p}{i}}, \frac{\detsymbol{i}{k}}{\detsymbol{m}{k}}\frac{L_{m,n}}{L_{i,j}}\right\} +
\left\{ \frac{\detsymbol{i}{m}\detsymbol{p}{k}}{\detsymbol{p}{m}\detsymbol{m}{k}} \frac{L_{m,n}}{L_{i,j}}, \frac{\detsymbol{i}{k}\detsymbol{m}{p}}{\detsymbol{m}{k}\detsymbol{i}{p}}\right\}
\\
&=&
\left\{ \frac{\detsymbol{p}{m}\detsymbol{k}{i}}{\detsymbol{k}{m}\detsymbol{p}{i}}, \frac{\detsymbol{i}{k}\detsymbol{p}{m}}{\detsymbol{i}{m}\detsymbol{p}{k}}\right\}
\\
&=&
\left\{ \frac{\detsymbol{k}{m}\detsymbol{p}{i}}{\detsymbol{p}{m}\detsymbol{k}{i}}, -\frac{\detsymbol{i}{m}\detsymbol{p}{k}}{\detsymbol{p}{m}\detsymbol{k}{i}}\right\}
\,,
\end{eqnarray*}
which is trivial because $\detsymbol{k}{m}\detsymbol{p}{i}-\detsymbol{i}{m}\detsymbol{p}{k}=\detsymbol{p}{m}\detsymbol{k}{i}$.
 This proves~(vii).
\end{proof}

\begin{proposition} \label{prop:basis}
Let $V$ be the subgroup of $K_2^T(C)$ generated by all the elements $\recsymbol{i}{j}{k}{l}{m}{n}$ and $\trisymbol{i}{j}{k}{l}{m}{n}$. Then $V$ is generated by the following elements:
\begin{eqnarray*}
&& \recsymbol{1}{1}{j}{2}{1}{m}, \qquad 1<j\le N_1, 1<m\le N_2; \\
&& \trisymbol{1}{1}{k}{l}{m}{n}, \qquad  2\le k < m \le N, 1\le l \le N_k, 1\le n \le N_m ; \\
&& \trisymbol{1}{j}{2}{1}{m}{n},
 \qquad 2\le j \le N_1, 3 \le m \le N,   1\le n \le N_m.
\end{eqnarray*}
\end{proposition}

\begin{proof}
We use the identities in Lemma~\ref{relations}.
First we prove that any $\recsymbol{i}{j}{k}{l}{m}{n}$ is a linear combination of these elements.
We can suppose $i<l$ by~(ii).
If $i=1$ and $l=2$, it is a consequence of~(iii) by letting $p=q=1$.
If $i=1$ and $l\neq 2$, it is a consequence of~(iv) by letting $p=2$ and $ q=1$, and using~(i).
If $i\neq 1$, then we reduce to the case $ i = 1 $ by using~(iv) with $p=q=1$.

Now we consider $\trisymbol{i}{j}{k}{l}{m}{n}$. By~(i), we can suppose $i<k<m$.
If $i=1$ and $k=2$, it is a consequence of~(v) by letting $p=q=1$.
If $i=1$ and $k\ne 2$, it is a consequence of~(vi) by letting $p=r=1$ and $q=2$.
If $i\neq 1$, it is a consequence of~(vii) by letting $p=q=1$.
\end{proof}

\begin{remark} \label{remark:genus}
Note that in Proposition~\ref{prop:basis}, $ V \subseteq K_2^T(C) $ is generated
by
\begin{equation*}
 (N_1 - 1)(N_2-1) + \sum_{2 \leq k < m \leq N} N_k N_m + (N_1-1) \sum_{m=3}^N N_m
\end{equation*}
elements. This number is exactly the same as the number in~\eqref{eqn:genus},
which is the genus of $ C $ if the affine curve defined by $ f(x,y) = 0 $
with $ f(x,y) $ as in~\eqref{eqn:curve} is smooth over the base field.
\end{remark}

We now also show that, using all symbols in $ K_2(F) $ with as entries functions
that have divisors supported in the $ \tP i j $, does not really lead
to a larger subgroup of $ K_2^T(C) $ than $ V $.

\begin{proposition}
Let $ A \subseteq K_2(F) $ be the subgroup generated by symbols
$ \{f_1,f_2\} $ where $ \divisor(f_l) $ is supported in $ \{\tP i j \}_{i,j} $. Then
there is a positive integer $ a $, depending only on $ N $ and the $ N_i $,
such that $ a(A \cap K_2^T(C)) $
is contained in the sum of $ V $ and $ K_2 $ of the base field.
\end{proposition}

\begin{proof}
From Remark~\ref{torsionremark} it is clear that
if $ \divisor(f_l) $ is supported in  $ \{ \tP i j \}_{i,j} $, then there is a fixed
positive integer such that $ f_l $ raised to this power is a product of powers of
the $ L_{i,j}/L_{1,1} $ and a non-zero constant.
Multiplying any element in $ A $ by the square of this positive integer, and
expanding, we see that the result can be expressed in terms of
elements of type $ T $,
$\left\{\frac{L_{1,j}}{L_{1,1}}, \frac{L_{k,l}}{L_{1,1}} \right\}$ with $k > 1$,
$\{c, \frac{L_{m,n}}{L_{1,1}} \}$ where $c$ is a non-zero constant,
and an element of $ K_2 $ of the base field.

Let us fix $ j $. Then for $ k_1, k_2 > 1  $ we have
$\left\{\frac{L_{1,j}}{L_{1,1}}, \frac{L_{k_1,l_1}}{L_{1,1}}\right\}-\left\{\frac{L_{1,j}}{L_{1,1}},\frac{L_{k_2,l_2}}{L_{1,1}}\right\}= \left\{\frac{L_{1,j}}{L_{1,1}}, \frac{L_{k_1,l_1}}{L_{k_2,l_2}}\right\} $.
If $k_1 = k_2= k$ this equals $\recsymbol{1}{j}{1}{k}{l_1}{l_2}$,
and, because it also equals
$  \left\{\frac{L_{1,j}}{L_{k_2,l_2}}, \frac{L_{k_1,l_1}}{L_{k_2,l_2}}\right\} -
\left\{\frac{L_{1,1}}{L_{k_2,l_2}}, \frac{L_{k_1,l_1}}{L_{k_2,l_2}}\right\}$,
for $ k_1 \ne k_2 $ it is the sum of two $T$-elements, elements
of the form $\left\{c, \frac{L_{m,n}}{L_{1,1}} \right\}$, and
an element in $ K_2 $ of the base field.
From $\l \prod_{i=1}^N \prod_{j=1}^{N_i} L_{i,j}=1$
we obtain
$ \sum_{k=1}^N \sum_{l=1}^{N_k} \left\{\frac{L_{1,j}}{L_{1,1}}, \frac{L_{k,l}}{L_{1,1}}\right\} = \left\{ \l , \frac{L_{1,j}}{L_{1,j}} \right\} + d \left\{ L_{1,1} , - L_{1,j} \right\} $,
where $ d = \sum_{i=1}^{N} N_i $.
Combining these two facts we see that
$ (d-N_1)\left\{\frac{L_{1,j}}{L_{1,1}}, \frac{L_{k,l}}{L_{1,1}}\right\}$
for $ k > 1 $ is the sum of
$R$-elements, $T$-elements, elements of the form $\left\{c, \frac{L_{m,n}}{L_{1,1}} \right\}$,
and an element in $ K_2 $ of the base field.
Hence there is a fixed positive integer such that if we multiply
an element in $ A $ by this integer, then the result can be expressed
in those four types of elements.

Now suppose this expression lies in $ K_2^T(C) $.
Collecting terms of the form $\left\{c, \frac{L_{m,n}}{L_{1,1}} \right\}$ for fixed $ m $
and $ n $, we may assume there is only one such term
$\left\{c_{m,n}, \frac{L_{m,n}}{L_{1,1}} \right\}$
for each pair
$ (m,n) $. It is trivial when $ (m,n) = (1,1) $.
For the other pairs, the divisors of the functions $ L_{m,n}/L_{1,1} $
are linearly independent as our earlier calculations show that
there is only one relation among the divisors of the $ L_{i,j} $,
which must correspond to the identity $ \l \prod_{i=1}^N \prod_{j=1}^{N_i} L_{i,j}=1$.
Because elements of type $ R $ and $ T $ are in $ K_2^T(C) $,
it follows that each $ c_{m,n} $ is a root of unity of order dividing
some fixed positive integer.
Multiplying the expression by this integer we obtain an element
in the sum of $ V $ and $ K_2 $ of the base field.
\end{proof}

\section{Integrality when $N=2$ or 3} \label{section:integrality}

In this section, we work over a number field, and we investigate
the integrality of the elements in~\eqref{eqn:rectangle} and~\eqref{eqn:triangle}
under certain conditions.
It should be noted that in the proof of Theorem~\ref{thm:main} below it is
crucial that the $ N $ different $ P_i $ at infinity in $ C $,
when viewed as sections of a regular, proper model $ \CC $, never meet,
and that the determinants $ [i,k] $ are units in the ring of
integers.
For general number fields this forces us to take $ N=2 $ or~3,
which we shall assume later in this section.

We begin by showing that, for $ C $ a regular, proper, geometrically irreducible curve
over some number field,
the subgroup $ \INTC $ defined in Section~\ref{introduction},
is independent of the choice of the regular, proper model $\CC$.
(On \cite[p.~344]{Jeu06}, where the base field was $ \Q $,
it was stated without proof that $\INTC /\torsion$, which was
denoted $ K_2(C;\Z) $, is independent of this choice.)

\begin{proposition} \label{modelchoice}
The subgroup $ K_2^T(\CC) $ of $ K_2^T(C) $
does not depend on~$ \CC $.
\end{proposition}

\begin{proof}
One sees  as on~\cite[p.~13]{Schneider} that the
image of $ K_2(\CC) $ in $ K_2(C) $ under localization is independent
of~$ \CC $.
Now consider the Gersten-Quillen
spectral sequence $ E_1^{p,q}(\CC) =  \coprod_{y \in \CC^{(p)}} K_{-p-q}(\F(y)) \Rightarrow K_{-p-q} (\CC) $,
where $ \F(y) $ denotes the residue field of $ y $.
It is compatible with the one for $ C $ under localization.
Then the image of $ K_2(\CC) $ in the quotient $ E_\infty^{0,-2}(C) $
of $ K_2(C) $ is also independent of $ \CC $, so the same holds
for the image of $ E_\infty^{0,-2}(\CC) \to E_\infty^{0,-2}(C) $. The latter
equals $ \ker(T) $, and we are done if we show that the former equals $ \ker(T_\CC) $.
This equality is implied by the surjectivity of the differential
$ E_1^{1,-3}(\CC) = \coprod_{\DD} K_2(\F(\DD)) \to E_1^{2,-3}(\CC) = \coprod_{y \in \CC^{(2)}} \F(y)^* $.
In order to see this, for a closed point $ y $ in $ \CC $ and
$ \b $ in $ \F(y)^* $, let
$ \Spec(\OO) \subset \CC $ contain $ y $, where $ \OO $ is an
order in a number field $ k_\OO $. Then as on \cite[p.~171]{dJ08}
one sees that there is an element in $ K_2(k_\OO) $
with image $ \b $ at $ y $ in the localization sequence for $ K_*'(\OO) $.
By compatibility (with shift in codimension)
of the Gersten-Quillen spectral sequence with
the inclusion $ \Spec(\OO) \to \CC $, this shows what we want.
\end{proof}

We now turn towards the question of integrality of the elements
in~\eqref{eqn:rectangle} and~\eqref{eqn:triangle} when $ N=2 $
or~3.
For those values of $N$,  by using a coordinate transformation
and replacing $ \l $, we can transform~\eqref{eqn:curve}
into
\begin{eqnarray} \label{eqn:integralcurve}
f(x,y) &=& \l  \prod_{i=1}^{N_1}(x+\a_i)\prod_{j=1}^{N_2}(y+\b_j)\prod_{k=1}^{N_3}(y-x+\c_k) - 1
\,,
 \quad \l \ne 0
\,,
\end{eqnarray}
where $ N_1 \ge N_2 \ge N_3 $, and we take $ N_3= 0 $ if $ N=2 $.
(Note that all $ \a_i $ are distinct, all $ \b_j $ are distinct,
and all $ \c_k $ are distinct.)
We again let $ C $ denote the non-singular model of the closure in
$ \P^2 $ of the affine curve defined by $ f(x,y) = 0 $
with $ f(x,y) $ as in~\eqref{eqn:integralcurve}. In this case, we have
$P_1=[0,1,0]$, $P_2=[1,0,0]$ and, if $ N = 3 $,  $P_3=[1,1,0]$.
If the affine curve
is non-singular, then $ C $ has genus $ g = N_1 N_2 + N_1 N_3 + N_2 N_3 - N_1 - N_2 - N_3 + 1 $,
also if $ N_3 = 0 $.  With our conventions $ g \ge 1 $ unless $ N_2 = 1 $ and $ N_3 = 0 $,
and $ g = 1 $ occurs only for $ N_1 = N_2 = 2 $ and $ N_3 = 0 $,
or $ N_1 = N_2 = N_3 = 1 $.

\begin{theorem}\label{thm:main}
Let all notation be as above, and assume $ \l $, and all the
$ \a_i $, $ \b_j $ and $ \c_k $ are
algebraic integers, with $ \l \ne 0 $.
Then the elements given by~\eqref{eqn:rectangle} and~\eqref{eqn:triangle}
are integral.
\end{theorem}

\begin{proof}
We first assume $ N_3 \ge 1 $, so that by Lemma~\ref{relations}
it suffices to prove this for the $ M = \{h_1,h_2\} $ with
$ h_1 = -\frac{x+\a_i}{y-x+\c_k} $
and $ h_2 = \frac{y+\b_j}{y-x+\c_k}$, which are of $ T $-type.

We can obtain a regular proper model $ \CC $ of $ C $ as follows.
Let $ \OO $ be the ring of algebraic integers in the base field.
We start with the arithmetic surface $ \CC' $ in $ \P_\OO^2 $ defined by
the homogeneous polynomial $ F $ of degree $ N_1 + N_2 + N_3 $
associated to~\eqref{eqn:integralcurve}.
We then first take the normalisation of $ \CC' $.
The resulting surface has generic fibre $ C $, and any singularities of
the surface are contained in its fibres at prime ideals $ \PP $
of $ \OO $. Those are resolved through iterated blow-ups, resulting
in our model $ \CC $.

Now let $ \DD $ be an irreducible component of some $ \CC_\PP $.
We have to show that $T_{\DD}(M) = 1$.  The image of $ \DD $
inside $ \CC_\PP' \subset \P_\OO^2 $ is either an irreducible component of $ \CC_\PP' $,
or a point of that curve.
Note that $ \CC_\PP' $ is defined in $ \P_{\OO/\PP}^2 $ by the reduction of $ F $ modulo $ \PP $.
So if $ \l $ is not in $ \PP $, then
$ \CC_\PP' $ does not contain the reduction of any of
the lines as a component.
If $ \l $ is in $ \PP $, then $ \CC_\PP' $ only has
the line at infinity as component.  In either case the $ h_i $
do not have a zero or a pole along any irreducible component of $ \CC_\PP' $.
So if the image of $ \DD $ in $ \CC_\PP' $ is not a point, then
$ v_\DD(h_1) = v_\DD(h_2) = 0 $, and $ T_{\DD} (M) = 1 $.

Now assume that $ \DD $ maps to a point of $ \CC_\PP' $.  If it
maps to an affine point of $ \CC_\PP' $, then $ \l $ is not in $ \PP $, $ h_1 $ and $ h_2 $
are regular at that point and attain non-zero values.
Therefore they are constant and non-zero on $ \DD $, hence
$ v_\DD(h_1) = v_\DD(h_2) = 0 $, and $ T_\DD(M) = 1 $.
The same
holds if $ \DD $ maps to a point at infinity in $ \CC_\PP' $ not equal to
the reductions of $ P_1 $, $ P_2 $ or $ P_3 $.
(Note this can only happen if $ \l $ is in $ \PP $.)
The remaining case is when it maps to the reduction of one of
the $ P_j $. By Lemma~\ref{relations}, the $ T $-elements are
alternating for renumbering the $ P_j $, so we may
assume $ \DD $ maps to $ [0,1,0] $ in $ \CC_\PP' $. But on $ \CC' $
the function $ h_2 $ is regular and equal to 1 at this point, so $ h_2 $ is
constant and equal to 1 along $ \DD $. Hence $ T_\DD(M) = 1 $
also in this case.

For $ N_3 = 0 $, all elements are of $ R $-type, and the proof (using
the same model $ \CC $) is similar.
\end{proof}

\section{When are the curves hyperelliptic for $ N=2 $ or~3?} \label{section:hyperelliptic}

In this section we work over an arbitrary base field of characteristic
zero.

Since a lot of work has been done to find elements of $K_2$ of
(hyper)elliptic curves \cite{Kimura,Jeu06}, we want to know when $C$ is
(hyper)elliptic. If it is not, it means that we found
curves that are geometrically more general.

With an eye on the restriction to $ N=2 $ or 3 in Theorem~\ref{thm:main},
and in order to avoid messy calculations in the proof of
Proposition~\ref{prop:hyperelliptic} below,
we impose the same restriction here, and may assume our curve
is defined by $ f(x,y) = 0 $ with $ f(x,y) $ as in~\eqref{eqn:integralcurve}.
According to Proposition~\ref{prop:hyperelliptic},
with $ N_3 = 1 $,
for $ N_2 = 1 $ we obtain (hyper)elliptic curves of arbitrary positive
genus, and for $ N_1 \ge N_2 \ge 2 $ we obtain non-hyperelliptic
curves of arbitrary composite genus.
Similarly, if $ N_3 = 0 $ and $ N_2 \ge 3 $, we find non-hyperelliptic
curves of arbitrary composite genus $ (N_1-1)(N_2-1) $.

\begin{proposition}\label{prop:hyperelliptic}
Suppose the affine curve defined by
$ f(x,y) = 0 $ with $ f(x,y) $ as in~\eqref{eqn:integralcurve}, is non-singular
and that $ C $ has positive genus.
Then $C$ is (hyper)elliptic if and only if either
$N_2=N_3=1$, or $N_2=2$ and $N_3=0$.
\end{proposition}

\begin{proof}
To prove this we shall use the following criterion of Max Noether (see \cite{Noether},
\cite[p.119]{Bernard}).
Suppose some $ h_i(x,y) \dd x $ span the space of holomorphic differentials on $ C $.
Consider the quadratic combinations $ \{h_i(x,y)h_j(x,y)\}_{i,j} $
of rational functions on~$ C $.
If $ C $ has genus $ g \ge 1 $, then they generate a space of
dimension at least $ 2g-1 $, and the
curve is (hyper)elliptic if and only if this dimension equals $2g-1$.

Assume $ N_3 \ge 1 $.
We can change the coordinates such that
all $\a_i,\b_j,\c_k \neq 0$. We first show that the forms
\begin{equation*}
\Omega_{i,j,k} = \frac{x^i y^j (x-y)^k \dd x}{\partial_yf(x,y)}
\qquad
(0\le i\le N_1-1,
 0\le j\le N_2-1,
 0\le k\le N_3-1)
\end{equation*}
generate $H^0(C,\Omega_C)$.

Because $\frac{\dd x}{\partial_yf(x,y)}=-\frac{\dd y}{\partial_xf(x,y)}$
and $\partial_x f(x,y),\partial_y f(x,y)$ do
not vanish at the same time by the non-singularity assumption,
the $\Omega_{i,j,k}$ are regular except possibly at the points
$ \tP i * $ above $P_1$, $P_2$ and $P_3$.
The order of $x$, $y$ and $x-y$ is
0, $-1$ and $-1$ at each of the points above $P_1$,
$-1$, 0 and $-1$ at the points above $P_2$, and
$-1$, $-1$ and $0$ at the points above $P_3$ respectively.
At the points above $P_1$, the order of $x+\a_i$ is
$N_2+N_3$ for exactly one of them and $0$ for the rest.
Hence the order of $\dd x$ is
$N_2+N_3-1$ at all points above $P_1$ and $-2$ at all points above $P_2$ or $P_3$.
The order of $\frac{1}{\partial_yf(x,y)}$ is $-1$, $N_1+N_3$ and $N_1+N_2$ at each
of the points above $P_1$, $P_2$ and $P_3$ respectively,
because
$\partial_yf(x,y) = \sum_{j=1}^{N_2} \frac{1}{y+\b_j}+\sum_{k=1}^{N_3} \frac{1}{y-x+\c_k}$.

\begin{table}[ht]
\caption{The order of some functions at the points $\tP i * $.}
\centering
\begin{tabular}{|l|l|l|l|}\hline
$ \vphantom{b^{b^{b^b}}} $
   & $\tP 1 * $ & $\tP 2 * $ & $\tP 3 * $\\ \hline
  $x$ & $0$ & $-1$ & $-1$ \\ \hline
  $y$ & $-1$ & $0$ & $-1$ \\ \hline
  $x-y$ & $-1$ & $-1$ & $0$ \\ \hline
  $\frac{1}{\partial_yf(x,y)}$ & $-1$ & $N_1+N_3$ & $N_1+N_2$ \\ \hline
\end{tabular}
\label{tab:order}
\end{table}

Therefore the order of $\Omega_{i,j,k}$ is $N_2+N_3-2-j-k$, $N_1+N_3-2-i-k$ and $N_1+N_2-2-i-j$ at each
of the points above $P_1,P_2$ and $P_3$ respectively.
Hence the $\Omega_{i,j,k} $ for $ 0\le i\le N_1-1, 0\le j\le N_2-1, 0\le k\le N_3-1$ are in $H^0(C,\Omega_C)$. Let $V$ be the vector space generated by these forms. Consider the polynomials $x^i y^j (y-x)^k$ for $ 0\le i\le N_1-1, 0\le j\le N_2-1, 0\le k\le N_3-1$.
As all $ N_i \ge 1 $, there are $(N_1-1)(N_2-1)(N_3-1)$ linear relations between these polynomials. Since the degree of these polynomials is less
than the degree of $f(x,y)$, all the linear relations between the
$\Omega_{i,j,k}$ come from the linear relations between these polynomials.
We conclude that the dimension of $V$ is $N_1N_2N_3-(N_1-1)(N_2-1)(N_3-1)=g$, so $V=H^0(C,\Omega_C)$.

In order to compute the dimension of the space spanned by the
quadratic combinations, let us start with the space $W$
of polynomials spanned by the
\begin{equation*}
x^iy^j(x-y)^k
\quad
(0\le i\le 2N_1-2, 0\le j\le 2N_2-2, 0\le k\le 2N_3-2)
.
\end{equation*}
It has dimension $ (2N_1-1)(2N_2-1)(2N_3-1) - (2N_1-2)(2N_2-2)(2N_3-2) $
as all $ N_i \ge 1 $.
On $ W $ we have to impose the relation given by $ f $.
If $f$ divides some $w $ in $ W$, so $w = uf$, then $H(w)=H(u)H(f)=x^{N_1}y^{N_2}(x-y)^{N_3}H(u)$,
where $H(\cdot)$ denotes the highest degree term of a polynomial. Since $w$ is in $ W$,
$H(w)$ is also in $ W$, and we can deduce that $H(u)$ is in the
space $U$ spanned by the
\begin{equation*}
x^iy^j(x-y)^k
\quad
(0\le i\le N_1-2, 0\le j\le N_2-2, 0\le k\le N_3-2)
.
\end{equation*}
Similarly, the highest term of $u-H(u)$ is in $ U$, so $u$ is in $ U$.
Hence the space of quadratic combinations of $ V $ has dimension
equal to $ \dim(W) - \dim(U) $.
Note that the dimension of $U$ is $(N_1-1)(N_2-1)(N_3-1)-(N_1-2)(N_2-2)(N_3-2)$ if
$N_1,N_2,N_3\ge 2$ and~0 otherwise.  In order to apply Noether's criterion, we compute
$ \dim(W) - \dim(U) - (2g-1) $, finding
\begin{equation*}
\begin{cases}
N_1N_2+N_1N_3+N_2N_3-N_1-N_2-N_3-1, \quad \text{if } N_1,N_2,N_3\ge 2, \\
2(N_1-1)(N_2-1), \quad \text{if } N_3 = 1.
\end{cases}
\end{equation*}
Because $ N_1 \ge N_2 \ge N_3 \ge 1 $, this is $0$ only if $ N_2 = N_3 = 1 $.

Similarly, if $N_3=0$ but $ N_2 \ge 2 $ (because the genus is positive), the forms
\begin{equation*}
\Omega_{i,j} = \frac{x^i y^j \dd x}{\partial_yf(x,y)}
\qquad
(0\le i\le N_1-2,
 0\le j\le N_2-2)
\end{equation*}
generate $H^0(C,\Omega_C)$.
Then we also consider the space of polynomials $W$ and relations $U$.
It is easy to see that
\begin{equation*}
\dim(W) - \dim(U) - (2g-1) =
\begin{cases}
(N_1-1)(N_2-1)-2, \quad \text{if } N_2 \ge 3,
\\
0, \quad \text{if } N_2 = 2.
\end{cases}
\end{equation*}
Because $ N_1 \ge N_2 $, this is $0$ if and only if $N_2=2$.
\end{proof}

In order to compare a curve $C$ satisfying the conditions in Proposition~\ref{prop:hyperelliptic}
with the (hyper)elliptic curves studied in~\cite{Jeu06},
we describe it as a ramified covering of $ \P^1 $ of degree~2, starting
from~\eqref{eqn:integralcurve}.

Firstly, if $N_2=2,N_3=0$, then we can assume that
$\a_1=\b_1=0 $ and $ \b_2=1$ by replacing
$x$ with $ x-\a_1 $ and  $y$ with $ (\b_2-\b_1) y - \b_1$.
Renumbering the non-zero $ \a_i $ and scaling $ \l $, we find $ C $ is defined by
\begin{equation} \label{eqn:hypercase1}
\l xy(y+1)\prod_{i=1}^{g}(x+\a_i) - 1 = 0
\,,
\end{equation}
where $g = N_1-1$ is the genus of the curve, and all $ \a_i $ are non-zero
and distinct.
Replacing $x$ with $ \frac{1}{x}$ and $y$ with
$ \frac{y+x^{g+1}}{\l\prod_{i=1}^{g}(\a_i x+1)}$, the equation of the curve becomes
\begin{equation} \label{eqn:hyper1}
y(y+2x^{g+1}+\l\prod_{i=1}^{g}(\a_i x+1))+x^{2g+2} = 0
\,.
\end{equation}
The element
$\left\{\frac{y+1}{y} , \frac{x+\a_i}{x}\right\}$
of type~\eqref{eqn:rectangle} for $C$ corresponding to~\eqref{eqn:hypercase1},
becomes
\begin{equation*}
 M_i
 = \left\{ \frac{y+x^{g+1}+\l\prod_{i=1}^{g}(\a_i x+1)}{y+x^{g+1}},  \a_i x+1 \right\}
 = \left\{-\frac{x^{g+1}}{y}, \a_i x+1\right\}
\end{equation*}
for the curve corresponding to~\eqref{eqn:hyper1}.
So $-2M_i$ is the element in \cite[Construction 6.11]{Jeu06} for the
factor $ \a_i x +1 $, which was considered in Example~10.8 of
loc.\ cit.\  for $ \l = \pm 1 $.

Secondly, if $N_2=N_3=1$, we can assume that $\b_1=\c_1=0$ by
replacing $x$ with $ x-\b_1+\c_1 $ and $ y $  with $ y - \b_1 $.
So $ C $ is defined by
\begin{equation}\label{eqn:hypercase2}
\l (y-x)y\prod_{i=1}^{g}(x+\a_i) - 1 = 0
\,,
\end{equation}
where $ g = N_1 $ and the $\a_i $ are distinct.
Replacing $x$ with $ \frac{1}{x}$ and $y $ with
$\frac{y+x^{g+2}}{\l x \prod_{i=1}^{g} (\a_ix+1)} + \frac{1}{x}$,
we find an equation
\begin{equation}\label{eqn:hyper2}
y(y+2x^{g+2}+\l\prod_{i=1}^{g}(\a_i x+1))+x^{2g+4} = 0
\,.
\end{equation}
The element
$$
\left\{\frac{y}{y-x}, -\frac{x+\a_i}{y-x} \right\}=\left\{\frac{y}{y-x}, -\frac{x+\a_i}{y-x}\right\}-\left\{\frac{y}{y-x},-\frac{x}{y-x}\right\} =\left \{\frac{y}{y-x}, \frac{x+\a_i}{x} \right\}
$$
of type~\eqref{eqn:triangle} for $C$ defined by~\eqref{eqn:hypercase2},
becomes $M_i=\{-\frac{x^{g+2}}{y}, \a_i x+1\}$ for the curve
defined by~\eqref{eqn:hyper2}.

Since $C$ defined by~\eqref{eqn:hyper1} is studied in
\cite[Example~10.8]{Jeu06},
we concentrate on $C$ defined by~\eqref{eqn:hyper2},
which can be of a different type (see Remark~\ref{different}).
Write
\begin{equation*}
  A(x) =  2x^{g+2}+\l\prod_{i=1}^{g}(\a_i x+1)
\,,
\end{equation*}
so that~\eqref{eqn:hyper2} becomes
\begin{eqnarray*}
  \left(y+\frac{A(x)}{2}\right)^2 - \frac{\l^2}{4}\prod_{i=1}^{g}(\a_i x+1)\prod_{j=1}^{g+2}(\mu_j x+1) = 0
\end{eqnarray*}
where $ 2x^{g+2}+A(x) = \l \prod_{j=1}^{g+2}(\mu_j x+1) $
for some $ \mu_j $,
which we assume to be in the base field by extending this if necessary.
We then have the following proposition similar to Propositions 6.3 and 6.14 of \cite{Jeu06}.

\begin{proposition} \label{prop:relation}
Let all notation be as above, with $\a_1,\dots,\a_g$ distinct.
Assume the resulting curve $ C $ has genus~$ g $ $($or, equivalently,
that $\mu_1,\dots, \mu_{g+2}$ are all distinct$)$.
With
\begin{eqnarray*}
\widetilde{M}_j = \left \{ -\frac{x^{g+2}}{y},  \mu_j x+1 \right \},
\quad
\M = \{-y, - x\},
\quad
\N =\left \{-\frac{x^{g+2}}{y}, - \frac{x^{g+2}}{\l} \right \}
\end{eqnarray*}
the following hold.
\begin{enumerate}
\item[(1)]
The $\widetilde{M}_j $ and $ \N$ are in $ K_2^T(C) $, and we
have
\begin{eqnarray*}
           2 \sum_{i=1}^{g} M_{i} + 2 \sum_{j=1}^{g+2} \widetilde{M}_j =  4 \N
\,,
\\
            \sum_{i=1}^g M_{i} =  \sum_{j=1}^{g+2} \widetilde{M}_j
.
\end{eqnarray*}
If $ \l^a = 1 $ then $ a \M $ is in $ K_2^T(C) $,
and $ 2 a \N = - 2 (g+2) a \M $.

\item[(2)]
If the base field is a number field, and $\l, \a_1,\dots, \a_g $ are algebraic
integers with $ \l $ a unit, then $ 2 \N $ and the $2\widetilde{M}_j $ are in
$\INTC$, and  if $ \l^a = 1 $ then $ a \M $ is in $ \INTC $.
\end{enumerate}
\end{proposition}

\begin{remark}
By the earlier coordinate transformations and Lemma~\ref{K2Tlemma},
the $ M_i $ are in $ K_2^T(C) $.
If the base field is a number field and $ \l, \a_1,\dots,\a_g $ are algebraic integers with $ \l \ne 0 $,
then they are in $\INTC$ by Theorem~\ref{thm:main}.
\end{remark}

\begin{proof}
First we look at the points at infinity of
$ C $ by taking
$x=1/\tilde{x}$ and $y=(\tilde{x}\tilde{y}-1)/\tilde{x}^{g+2}$,
so~\eqref{eqn:hyper2} gives
\begin{equation}\label{eqn:infinitymodel}
\tilde{y}^2 + \l(\tilde{x} \tilde{y}-1) \prod_{i=1}^{g}(\tilde{x}+\a_i)= 0
\,.
\end{equation}
If all $ \a_i \ne 0$ then there are two points at infinity,
namely $(0, \pm \sqrt{\l \prod_{i=1}^{g}\a_i})$, which we denote
by $\infty$ and $\infty'$. If some $ \a_i = 0 $ then there
is only one point at infinity, and we let $ \infty $ and $ \infty' $
both denote this point.

Let $P_{\mu_j}$ be the point $(-\mu_j^{-1},(-\mu_j^{-1})^{g+2})$,
and let $O$ and $O'$ be the points $(0,0)$ and $(0,-\l)$ respectively.
Then
\begin{eqnarray*}
\divisor(x) & = & (O) + (O') - (\infty) - (\infty'), \\
\divisor(y) &=& (2g+4)(O) - (g+2)(\infty) - (g+2)(\infty'), \\
\divisor(\mu_j x+1) & = & 2 (P_{\mu_j})-(\infty)-(\infty'), \\
\divisor \left(-\frac{x^{g+2}}{y}\right) &=& (g+2)(O') - (g+2)(O).
\end{eqnarray*}

Since $-\frac{x^{g+2}}{y}$ and $ \mu_j x+1$ only have zeros and poles at
$P_{\mu_j}, O, O', \infty$ and $\infty'$, the tame symbol of $\widetilde{M}_j$ is
trivial except at these points.
We discuss the case when all $ \a_i \ne 0 $, the other case being
simpler.
Then we have
\begin{eqnarray*}
T_{P_{\mu_j}}(\widetilde{M}_j) & = & (-1)^0 \left.\left(-\frac{x^{g+2}}{y}\right)^2\right|_{P_{\mu_j}} = 1, \\
T_{O}(\widetilde{M}_j) & = & (-1)^{0} \left.\frac{1}{(\mu_j x+1)^{-(g+2)}}\right|_{O}  = 1 , \\
T_{O'}(\widetilde{M}_j) & = & (-1)^{0} \left.\frac{1}{(\mu_j x+1)^{(g+2)}}\right|_{O'}  = 1 , \\
T_{\infty}(\widetilde{M}_j) & = & (-1)^0 \left.\frac{-y}{x^{g+2}}\right|_{\infty} = \left.1-\tilde{x}\tilde{y}\right|_{\infty}=1,\\
T_{\infty'}(\widetilde{M}_j) & = & (-1)^0 \left.\frac{-y}{x^{g+2}}\right|_{\infty'} = \left.1-\tilde{x}\tilde{y}\right|_{\infty'}=1.
\end{eqnarray*}

For $ \M $, we only need to calculate the tame symbol $O$, $ O'$, $ \infty$ and $\infty'$,
which gives
\begin{eqnarray*}
T_{O}(\M) & = & (-1)^{2g+4} \left.\frac{-y}{(-x)^{2g+4}}\right|_{O}  =  \l^{-1}, \\
T_{O'}(\M) & = & (-1)^0 \left. -y\right|_{O'}  = \l , \\
T_{\infty}(\M) & = & (-1)^{g+2}\left.\frac{(-y)^{-1}}{(-x)^{-(g+2)}}\right|_{\infty} = \left.\frac{1}{1-\tilde{x}\tilde{y}}\right|_{\infty}=1,\\
T_{\infty'}(\M) & = & (-1)^{g+2}\left.\frac{(-y)^{-1}}{(-x)^{-(g+2)}}\right|_{\infty'} = \left.\frac{1}{1-\tilde{x}\tilde{y}}\right|_{\infty'}=1.
\end{eqnarray*}
Here the first calculation uses that
$-x^{2g+4}/y =  y+2x^{g+2}+\l\prod_{i=1}^{g}(\a_i x+1) $
by~\eqref{eqn:hyper2}.
From this our claim for $ a \M $ is clear.
One checks similarly that $ \N $ is in $ K_2^T(C) $.
That $ 2 a \N = - 2 (g+2) a \M $ is easily checked.

The remaining relations in~(1) can be proved as in \cite[Propositions 6.3 and~6.14]{Jeu06}.
The statement about integrality can be proved as in \cite[Theorem 8.3]{Jeu06} except we use
the model defined by~\eqref{eqn:infinitymodel} at the points at infinity.
Note that the $ \mu_j $ are algebraic integers by our assumption
on $ \l $ and the $ \a_i $.
\end{proof}

\begin{remark} \label{different}
If the base field is $ \Q $ and $ \l $ is fixed, it is easy to give
examples of families defined by~\eqref{eqn:hyper2}
with exactly $ g $ rational Weierstrass points. Thus, the curves
are in general really different from those studied in \cite{Jeu06},
which all had at least $ g+1 $ rational Weierstrass points over
the base field.
\end{remark}

\section{Linear independence of the elements} \label{section:linearindependence}

In this section, we work over the base field $ \C $, and
look at the curves $ C $ constructed in
Section~\ref{section:construction} with arbitrary $ N \ge 2 $
as a family with one parameter $ t = 1/\l $.
We give a description of part of a basis of
$H_1(C(\C);\Z)$ when $ t $ is close to zero.
We compute the limit behaviour of the pairing~\eqref{eqn:pairing}
of those elements
with the elements in Proposition~\ref{prop:basis}, showing in
particular that the latter are usually linearly independent.
If all coefficients are in $ \R $ then those elements form a basis
of $H_1(C(\C);\Z)^-$, and we obtain a non-vanishing result for the
regulator. Of course, those results also apply to subfields of
$ \C $ or $ \R $ when $ t $ goes to 0, or fields that can be
embedded this way.
In particular, working over 
a number field, from Theorem~\ref{thm:main} we find examples
of curves of genus $ g $ with $ g $ independent integral elements
in $ K_2^T(C) $.
If this number field is $ \Q $, then we get examples with as many
independent integral elements as needed in Beilinson's conjecture.

For fixed $a_i$,  $ b_i$ and $ c_{i,j}$, the
normalisations
of the projective closures of the curves
defined by $ f(x,y) = 0 $ with $ f(x,y) $ as in~\eqref{eqn:curve}
form a family of curves $ C_t $ in the parameter $t = 1/\l$.
The affine part of $ C_t $ is defined by
\begin{equation} \label{eqn:tcurve}
 \prod_{i=1}^N \prod_{j=1}^{N_i} L_{i,j} -t = 0
\end{equation}
for $ t \ne 0 $, but we shall use the resulting curve $ C_0 $
for $ t=0 $ extensively.

By Lemma~\ref{lemma:smooth}, if we take $ t $ in a small enough
disc $ D $ around~0, then the resulting
curves $ C_t $ for $ t \ne 0 $  are regular.
So if we let $ X $ be the complex manifold with points
$ \{ C_t(\C) \}_{t \in D} $, then
the fibres $ X_t = C_t(\C) $ for $ t \ne 0 $ are Riemann surfaces of genus
$ g $ given by~\eqref{eqn:genus}.
The fibre $ X_0 $ consists of $\P_\C^1$s corresponding
to the affine lines defined by $L_{i,j} = 0$, and their points of intersection
correspond to the intersection points of those affine lines.

For the remainder of this section we make the following assumption.

\begin{assumption} \label{assumption}
No three of the lines defined by $ L_{i,j} = 0 $ meet in an affine point.
\end{assumption}

We shall make certain elements in $ H_1(X_t;\Z) $ with $ t \ne 0 $
more explicit than in the general theory
(see, e.g., \cite[Expos\'e~XV, Th\'eor\`eme~3.4]{SGA72}),
which will enable us to carry out the necessary calculations.

Let $ d = \sum_{i=1}^{N} N_i $ be the degree of~\eqref{eqn:tcurve}.
Changing coordinates, we may assume that for each
value of $ x $ in $ \C $ there are $ d $ or $ d-1 $ affine
points $ (x,y) $ in $ X_0 $, with the latter case
occurring for the $ x $-coordinate of the intersection point
of two of the lines.
Let $ \pi $ be the projection from the affine part of $ X $ to $ \C \times D $
obtained by mapping $ (x,y,t) $ to $ (x,t) $.
It is easy to see from~\eqref{eqn:tcurve} and Assumption~\ref{assumption}
that for $ t=0 $, $ \pi $ is unramified except at points $ (s, 0) $
with $ s $ the affine intersection point of two of the lines.

Now fix such an $ s $.
If we take a loop $ \c' = \{ z \text{ in } \C \text{ with } |z-s| = r \} $
for some $ r >0 $ small enough,
then $ \pi $ gives an unramified covering of $ \c' \times \{0\} $.
The same holds for $ \c' \times D $ if we shrink $ D $ to a small
enough disk containing $ 0 $.
Choose a parametrisation $ [0,1] \to \c' $ mapping 0 and
1 to some $ Q $.
Using the identity map on $ D $, we obtain a map
$ [0,1] \times D \to  \c' \times D $, which we can lift to a
map $ [0,1] \times D \to  X $
as $ D $ is simply connected.
From the restriction of this lift to $ \{0\} \times D  $ and $ \{1\} \times D $
we obtain sections of the covering by $ \pi $ of $ \{Q\} \times D $,
which may land in different branches of~$ \pi $.
But those sections either agree 
for all $ t $ in $ D $, or for no $ t $ in $ D $.
In particular, the resulting paths $ [0,1] \times \{t\} \to X_t $ are
all closed, or all non-closed.

Now let $ L $  be one
of the two affine lines intersecting at~$ s $, so $ \pi $
maps $ L \times\{0\} $ bijectively to $ \C \times \{0\} $.
This way $ \c' \times \{0\} $ lifts to
a loop $ \c_0 $ in $ L \times \{0\} $.
Parametrizing $ \c' $ as above, we can choose the lift of
$ [0,1] \times D \to X $ of $ [0,1] \times D \to \c' \times D $
that for $ t=0 $ coincides with the parametrisation of $ \c_0 $.
As $ \c_0 $ is a loop, we obtain a family of loops $ \c_t $ in $ X_t $.
Shrinking $ \c' $ and $ D $ if necessary, we can get
all points in $ \c_t $ to be as close to $ (s,t) $ as we want.

This way, with $ t $ small enough, for every point of intersection $ s $ in $ X_0 $,
we obtain a family of loops $ \c_s = \c_{s,t} $ with $ \c_{s,t} $ in $ X_t $.
(If $ t $ is clear from the context, we shall sometimes write $ \c_s $ where we mean the loop $ \c_{s,t} $.)
We shall choose $ g $ of them that give part of a basis of
$H_1(X_t;\Z)$ for $t\ne 0$ small enough, and later describe the
limit behaviour of the regulator pairing of such an element and
an element in Proposition~\ref{prop:basis}.

For this choice, let $ P_{i,j; k,l} $ be the (affine) point 
where $ L_{i,j} $ and $ L_{k,l} $ for $ i \ne k $ vanish,
and consider the set
\begin{equation*}
 S
=
 \{P_{i,j;k,l}| 1\le i<k \le N, 1\le j\le N_i, 1\le l\le N_k, (i,j)\ne (1,1), (i,k,l)\ne (1,2,1)\}
\,.
\end{equation*}
It consists of all points of intersection in $ X_0 $, except
for those where either $L_{1,1}$ vanishes
or both $L_{2,1}$ and some $L_{1,j}$ with $ 2\le j \le N_1$ vanish.
(For example, for $N=3$, if in~Figure \ref{fig:badfibre} the $L_{1,j}$,
$L_{2,l}$, and $L_{3,n}$ vanish at the diagonal, horizontal, and vertical lines respectively,
then the elements of $ S $ are the thick points.)
There are $\sum_{1\le i<k \le N}N_i N_k - \sum_{1\le k \le N} N_k + 1=g$
points in~$ S $.

\setlength{\unitlength}{1mm}
\begin{figure}[h]
\centering
\begin{picture}(60,60)
\multiput(-5, 5)(0, 10){3}{\line(1, 0){60}}
\multiput(30, 0)(10, 0){3}{\line(0, 1){60}}
\put(0,0){\line(1,1){60}}
\put(-10,0){\line(1,1){60}}
\multiput(0, 5)(0, 10){3}{
\multiput(30, 0)(10, 0){3}{\circle*{2}}
}
\multiput(15, 15)(10, 10){2}{\circle*{2}}
\multiput(30, 30)(10, 10){3}{\circle*{2}}
\end{picture}
\caption{Configuration of $ X_0 $, and the set $ S $.}
\label{fig:badfibre}
\end{figure}
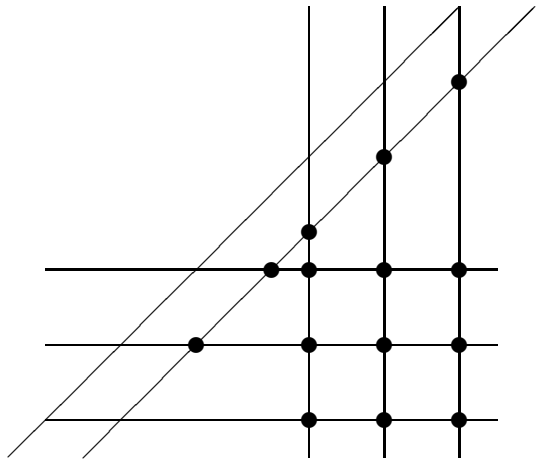

\begin{lemma} \label{lemma:basis}
With notation and assumptions as above, for $ t \ne 0 $ and $ t $ small enough:
\begin{enumerate}
\item[(1)]
$ \{\c_s\}_{s \in S} $ can be complemented to a basis of $ H_1(X_t;\Z) $;

\item[(2)]
if all $ a_i $, $ b_i $, $c_{i,j} $, and $ t $ are in $ \R$, then
$ \{\c_s\}_{s \in S} $ is a basis of $H_1(X_t;\Z)^- $.
\end{enumerate}
\end{lemma}

\begin{proof}
If all $a_i$,  $b_i$, $c_{i,j} $ and $ t $ are real, then
$ \c_{s,0} $ is anti-invariant under complex conjugation
by construction and the same holds for $ \c' $. By the construction
of the family $ \c_{s,t} $ by lifting, $\c_{s,t}$ is also anti-invariant.
So the $\c_s$ are in $H_1(X_t;\Z)^{-}$, and
part~(2) follows from part~(1)
as $ H_1(X_t; \Z)^- $ has rank~$ g $.

In order to prove part~(1), we construct
another family of loops $ \{\delta_s \}_{s \in S} $. Note that, for
$ s $ the affine point of intersection of the two lines
defined by $ l_1 = 0 $ and $ l_2 = 0 $,
\eqref{eqn:tcurve} is of the form $ l_1 l_2 h - t = 0 $
where $ h(s) \ne 0 $.  One then sees easily that for $ t \ne 0 $ but small,
the point $ s $ splits into two ramification points,
and we can parametrise one of the two ramification points
for $ t $ in a suitable circle sector of $ D $.
Using this, one sees that any loop $ \delta $ in $ X_0 $ that
is obtained by connecting
distinct affine intersection points of lines using paths in the lines,
can be extended to a continuous family of loops $ \delta = \delta_t $
for $ t $ in such a sector.

We apply this to loops on $X_0$ that we denote as
\begin{eqnarray*}
&&(L_{1,1}, L_{k,l}, L_{m,n}), \quad 2\le k < m \le N, 1\le l \le N_k, 1\le n \le N_m, \\
&&(L_{1,1}, L_{2,1}, L_{1,j}, L_{2,n}), \quad 1<j\le N_1, 1<n\le N_2, \\
&&(L_{1,1}, L_{2,1}, L_{1,j}, L_{m,n}),
    \quad 2\le j \le N_1  , 3\le m \le N, 1\le n \le N_m
.
\end{eqnarray*}
Here with $ (l_1,\dots,l_p) $ we mean a loop obtained by starting
with as vertices the $ p $ intersection points of pairs of
lines defined by $ l_i = 0 $ and $ l_{i+1} = 0 $ (indices modulo~$ p $),
and connecting two consecutive vertices by
a path in the line containing both vertices.

It is easy to see that among those vertices there
is a unique $ s $ in $ S $, namely $ P_{k,l;m,n} $ in the first
case, $ P_{1,j;2,n} $ in the second, and $ P_{1,j;m,n} $ in the
third. We can choose the connecting paths to avoid all other
points in $ S $, so that on $ X_0 $ we have a loop $ \delta_s $
that contains $ s $ but no other points in $ S $, and on suitable
sectors in $ D $ all those $ \delta_s $ can be deformed as above.
Note that we defined exactly~$ g $ loops~$ \delta_s $.

For $ s $ and $ s' $ in $ S $, and working in a suitable sector
of a small enough $ D $ all the time, $\c_s$ and $\delta_{s'}$ in $ X_t $ can only intersect
if $ s = s' $.
Considering the different branches
in $ X $ for $ \pi $ above points close to $ (s,0) $, we
see that $ \c_s $ and $ \delta_s $ meet exactly once in
$ X_0 $, hence also in $ X_t $ with $ t \ne 0 $.
Changing the orientation of $\delta_s$ if necessary, we can assume
that $\c_{s} \cap \delta_{s'}$ equals 1 if $s=s'$ and 0 otherwise.
Since $ \c_s \cap \c_{s'} = 0 $ always, the intersection matrix of
$ \{\c_s\}_{s \in  S} \cup \{\delta_s\}_{s \in S}$
on $ X_t $ with $ t \ne 0 $
is of the form
\begin{equation*}
\left(
  \begin{array}{cc}
    0 & I_g \\
    -I_g & * \\
  \end{array}
\right)
\,.
\end{equation*}
As this has determinant 1, the $\c_s $ and $ \delta_s$ for $ s $
in $ S $ form a basis of $H_1(X_t;\Z)$.
\end{proof}

We want to establish a limit formula for the integral over $ \c_t $ of the regulator 1-form
$ \eta $ as in~\eqref{eqn:eta} obtained from the elements constructed in Section~\ref{section:construction}.
For this we shall need the following two lemmas.

\begin{lemma} \label{lemma:di}
Let $ V\subseteq \C$ be open and let $\varphi = \mu_1(u,t) \dd u + \mu_2(u,t)\dd t+ \mu_3(u,t)\dd\ol{t}$
be a $1$-form on $\R \times V$,
periodic in $u$ with period $p$. If
$\dd \varphi = \nu_1(u,t)\dd u \wedge \dd t + \nu_2(u,t)\dd u \wedge \dd \ol{t}+ \nu_3(u,t)\dd t \wedge \dd \ol{t}$,
then for any $c$ in $\R$ we have
\begin{equation*}
   \frac{\partial}{\partial t}\int_{u=c}^{c+p} \varphi
=
  -\int_{u=c}^{c+p}\nu_1(u,t) \dd u
\qquad\text{ and }\qquad
   \frac{\partial}{\partial \ol t } \int_{u=c}^{c+p} \varphi
=
  -\int_{u=c}^{c+p}\nu_2(u,t) \dd u
\,.
\end{equation*}
\end{lemma}

\begin{proof}
We note that
$$    \frac{\partial}{\partial t}\int_{u=c}^{c+p} \varphi
=
   \frac{\partial}{\partial t}\int_{u=c}^{c+p}\mu_1(u,t) \dd u
=
   \int_{u=c}^{c+p}\frac{\partial}{\partial t}\mu_1(u,t) \dd u
=
  -\int_{u=c}^{c+p}\nu_1(u,t) \dd u
\,.
$$
The first two identities here are clear, and the last one holds because
$\nu_1(u,t) = -\frac{\partial}{\partial t}\mu_1(u,t)+\frac{\partial}{\partial u}\mu_2(u,t)$,
and $\int_{u=c}^{c+p}\frac{\partial}{\partial u}\mu_2(u,t) \dd u = \mu_2(c+p,t) - \mu_2(c,t)=0$.
The other statement is proved similarly.
\end{proof}

Assume that in a neighbourhood of $ (0,0,0) $ in $ \C^3 $, we
have a surface $ Y $ defined by $x y h(x,y) - t = 0 $, with $ h(x,y) $
holomorphic around $ (0,0) $ and $ h(0,0) \ne 0 $.
Assume that there is a family of loops $\c_t$ in the fibres $ Y_t $,
with $ \c_0 $ a clockwise simple loop around $ 0 $ in the $ x $-axis in $ Y_0 $.
Furthermore, let $ u $ and $ v $ be holomorphic functions in $ x $ and $ y $
around $ (0,0) $ that do not vanish at $ (0,0) $.

\begin{lemma} \label{lemma:limit}
Let $Y$, $ u $, $ v $ and the $\c_t$ be as above and assume $ \c_0 $ is sufficiently small.
For
integers $ a $ and $ b $,
Let $ \psi(u x^a, v y^b)$ be the 1-form
$ \log|u x^a| \dd \arg(v y^b) - \log|v y^b| \dd \arg(u x^a)$
on an open part of $ Y \setminus Y_0 $, and let
$F(t)=\int_{\c_t} \psi(u x^a, v y^b) $ for $ t \ne 0 $
sufficiently small.
Then
$ F(t) = 2 \pi a b \log|t| + \re(H(t)) $
for a holomorphic function $ H(t) $ around $ t=0 $. In particular,
\begin{equation*}
\lim_{\abs{t}\to 0} \frac{F(t)}{\log \abs{t}} = 2\pi ab
\,.
\end{equation*}
\end{lemma}

\begin{remark} \label{rem:thought}
Note that $ \psi_{|Y_t} $ is $ \eta(u x^a_{|Y_t}, v y^b_{|Y_t}) $ with $ \eta $
as in~\eqref{eqn:eta}. Therefore the lemma should be thought
of as describing the limit
behaviour of $ 2 \pi \langle \c_t , \{u x^a_{|Y_t}, v y^b_{|Y_t}\} \rangle $
with
$ \langle \,\cdot\, , \,\cdot\, \rangle $ as in~\eqref{eqn:pairing}.
\end{remark}

\begin{proof}
First note that $ x $ and $ y $ are non-zero on $ Y \setminus Y_0 $
so that the definition of $ F(t) $ makes sense because of our
assumptions on $ u $, $ v $ and $ \c_0 $.
If we can
write $\dd \psi(u x^a,v y^b) = \omega_1 \wedge \dd t + \omega_2 \wedge \dd \ol{t}$
with 1-forms $\omega_1$ and $\omega_2$
on a suitable open part of $ Y \setminus Y_0 $
containing the $ \c_t $ for $ t\ne 0 $ small enough,
then applying Lemma~\ref{lemma:di} to a parametrisation of
$ \{\c_t\}_{t \in D} $ gives
$ \frac{\partial F(t)}{\partial t} =  -\int_{\c_t}  \omega_1$.

In order to calculate $ \omega_1 $, we notice that on $ Y \setminus Y_0 $
where $ h $ does not vanish we have the identity
\begin{equation*}
\frac{h_x \dd x + h_y \dd y}{h}+\frac{\dd x}{x}+\frac{\dd y}{y}=\frac{\dd t}{t}
\,.
\end{equation*}
With $ h_1 = x h_x h^{-1} + 1 $ and $ h_2 = y h_y h^{-1} + 1 $
we have $ \frac{h_1}{x} \dd x + \frac{h_2}{y} \dd y = \frac{\dd t}{t} $.
Since $ h_2 $ is a holomorphic function of $ x $ and $ y $
around $ (0,0) $ with $ h_2(0,0) = 1 $, close enough to $ (0,0,0) $
on $ Y\setminus Y_0 $ we have
$ \dd x \wedge \dd y =  \frac{xy}{h_2} \frac{\dd x}{x} \wedge \frac{\dd t}{t} $
and
\begin{eqnarray*}
\dd\log(u x^a)\wedge \dd\log(v y^b) &=& (\dd\log u + a \dd\log x)\wedge (\dd\log v + b \dd\log y)
\\
&=&
 \frac{1}{h_2} \big(a b + \frac{b x u_x}{u} + a \frac{y v_y}{v} \big) \frac{\dd x}{x} \wedge \frac{\dd t}{t}
\\
&=&
(a b + x h_3 + y h_4 ) \frac{\dd x}{x} \wedge \frac{\dd t}{t}
\end{eqnarray*}
with $h_3(x,y)$ and $h_4(x,y)$ holomorphic functions around $ (0,0) $
as one sees from $ \frac{1}{h_2} = 1 - \frac{y h_y}{h h_2} $
and the properties of $ u $, $ v $, $ h $ and $ h_2 $.

Denote $(a b + x h_3 + y h_4 ) \frac{\dd x}{x}$ by $\omega$.
Clearly,
\begin{equation*}
  \dd \psi(u x^a, v y^b) = \im (\dd\log(u x^a)\wedge \dd\log(v y^b))
=
  \frac{1}{2i} \omega \wedge \frac{\dd t}{t} - \frac{1}{2i} \ol{\omega} \wedge \frac{\dd\ol{t}}{\ol{t}}
\,,
\end{equation*}
so that we can take
$\omega_1 = \frac{1}{2it}{\omega}$.
Then we have
\begin{equation*}
   \frac{\partial F(t)}{\partial t}
=
 - \int_{\c_t} \omega_1
=
 - \int_{\c_t} \frac{1}{2it} (a b + x h_3 + y h_4) \frac{\dd x}{x}
=
 \frac{\pi ab}{t} + h_5(t)
\end{equation*}
where $h_5(t)$ is a holomorphic function around $t = 0$.
Namely,
$ \int_{\c_t} (x h_3 + y h_4) \frac{\dd x}{x} $ is continuous
in $ t $, and 
vanishes for $ t=0 $ as $ y=0 $ on~$ \c_0 $, $ h_3 $ is holomorphic
around $ (0,0) $, and $ \c_0 $ is sufficiently small.
Moreover, our earlier expression for $ \dd x \wedge \dd y $
on $ Y \setminus Y_0 $ close to $ (0,0,0) $ shows that 
the holomorphic 2-form $ \dd ((x h_3 + y h_4) \frac{\dd x}{x}) $
there is of the form $ \frac{g(x,y)}{t} \dd x \wedge \dd t $
with $ g(x,y) $ holomorphic around $ (0,0) $.
Because $ \c_t \subset Y_t $, the $ \dd t $ will remain unchanged
under pullback under a parametrisation of the $ \{\c_t\}_t $ for $ t \ne 0 $
small enough.
From Lemma~\ref{lemma:di} we then see that
$\frac{\partial}{\partial \ol t}  \int_{\c_t} (x h_3 + y h_4) \frac{\dd x}{x} = 0 $
for such~$ t $ because $ \nu_2 $ is identically 0.
Therefore
$\int_{\c_t} (x h_3 + y h_4) \frac{\dd x}{x} $ is holomorphic for~$ t $
small enough, and vanishes for $  t = 0 $.

Since $\frac{\partial \log\abs{t}}{\partial t} = \frac{1}{2t}$, we have
$\frac{\partial (F(t) - 2\pi ab\log\abs{t})}{\partial t} = h_5$ around $t=0$.
Both $F(t)$ and $\log\abs{t}$ are real-valued, hence also
$\frac{\partial (F(t) - 2\pi ab\log\abs{t})}{\partial \ol{t}} = \ol{h_5}$
around $t=0$.
Therefore $F(t) - 2\pi ab\log\abs{t} = 2 \re(h_5(t)) + c $
for a real constant~$ c $.
\end{proof}

We now return to the curves $ X_t $ with $ t \ne 0 $ in $ D $
and $ D $ sufficiently small.
Here we have  classes $ \a_1, \dots, \a_g $
in $K_2^{T}(C)/\torsion$ from Proposition~\ref{prop:basis}, and
we pair those under the regulator pairing~\eqref{eqn:pairing} with
the $ g $ loops $ \c_s $ of Lemma~\ref{lemma:basis}.
We observed that we may assume that the points on $ \c_{s,t} $
in the family $ \c_s $ are as close to $ (s,t) $ as we want.
We then have the following limit result.

\begin{theorem} \label{thm:limit}
Let $X_t$ be defined by~\eqref{eqn:tcurve}, and assume no three
$ L_{i,j} $ meet at an affine point.
If the $a_i$, $ b_i $ and  $c_{i,j}$ are fixed,
then
\begin{equation*}
\lim_{t\to 0}\frac{ \det \left( \langle \c_s , \a_j \rangle \right) }{ \log^g \abs{t} }= \pm 1
\,.
\end{equation*}
\end{theorem}

Before giving the proof of the theorem, we have the following immediate corollary,
which is the main result of this section.
Note that we appeal to Theorem~\ref{thm:main}, and implicitly
embed the number field that is the base field into the complex
numbers in order to apply Theorem~\ref{thm:limit}.

\begin{corollary} \label{cor:int}
Let $C$ be defined by~\eqref{eqn:curve} and assume no
three of the lines defined by $ L_{i,j} =0 $ meet in an affine point.
If the $a_i$, $ b_i $  and $c_{i,j}$ are fixed, and $\abs{\l}\gg 0$, then
the elements are independent in $K_2^T(C)$.
In particular, for $C$ defined by~\eqref{eqn:integralcurve} with
$ |\l| \gg 0 $ and satisfying the condition of Theorem \ref{thm:main}, we have $g$ independent elements in $\INTC$,
where $ g $ as in~\eqref{eqn:genus} is the genus of~$ C $.
\end{corollary}

\begin{proof}[Proof of Theorem \ref{thm:limit}]
To every $s' = P_{i,j;k,l} $ in $ S$ with $ i < k $, we associate
an element $M_{s'}$ in $K_2^T(C)$, namely
\begin{equation*}
M_{s'} =
\begin{cases}
  \recsymbol{1}{j}{1}{2}{l}{1} & \text{if } i=1, k=2,\\
  \trisymbol{1}{1}{i}{j}{k}{l} & \text{if } 2\le i < k \le N, \\
  \trisymbol{1}{j}{2}{1}{k}{l}-\trisymbol{1}{1}{2}{1}{k}{l} & \text{if } i=1, k>2.
\end{cases}
\end{equation*}
Note that the elements above are the same as the elements in Proposition \ref{prop:basis}
in the first two cases,
and in the last case it is an element minus an element of the second case.
So the regulator of these elements is the same as the regulator of those
in Proposition \ref{prop:basis},
and the theorem will be proved if we show that
\begin{equation*}
   \lim_{t \to 0} \frac{ \langle \c_s , M_{s'} \rangle}{\log\abs{t}}
=
   \begin{cases}
        \pm 1, & \text{if } s=s', \\
        0, & \text{otherwise}.
   \end{cases}
\end{equation*}

Recall that
$ \langle \c_s , M_{s'} \rangle = \frac{1}{2\pi} \int_{\c_{s}}{\eta(M_{s'})} $,
with $ \eta(M_{s'}) $ as in~\eqref{eqn:eta}.
We shall compute this limit using Lemma~\ref{lemma:limit} (cf.~Remark~\ref{rem:thought}).
As we mentioned just before the theorem, for $ s $ in $ S $ we can assume that the
points in the loop $ \c_{s,t} $ are all as close to $ (s,t) $
as we want. If $ L_1 = 0 $ and $ L_2 = 0 $
are equations of distinct lines through $ s $, we can change
coordinates so that the lines are $ x=0 $ and $ y=0 $, and the
loops corresponding to $ \c{_s,t} $ are as close to $ (0,0,t) $
as we want.
Since $ \psi(x,y) = \eta(x,y) $ with $ \eta $ as in~\eqref{eqn:eta},
under such a coordinate change we obtain from Lemma~\ref{lemma:di} that
\begin{equation*}
\frac{1}{2\pi}\lim_{t\to 0} \frac{\int_{\c_{s}}{\eta(L_{i,j},L_{k,l})}}{\log \abs{t}}
=
\begin{cases}
  \pm 1 \text{ if } L_{i,j}\text{ and } L_{k,l} \text{ vanish at } s, \\
  0 \text{ otherwise}.
\end{cases}
\end{equation*}
Now take $s'=P_{i,j;k,l}$ in $ S $ with $ i < k $. Since $ S $
does not contain the point at which both $L_{1,1}$ and $ L_{i,j} $ for $ i>1 $
vanish,
nor the point at which both $L_{2,1}$ and $ L_{1,j} $ vanish,
we see by expanding $\eta({M_{s'}})$ from~\eqref{eqn:rectangle} and~\eqref{eqn:triangle}
that
for $i=1,k=2,$
\begin{equation*}
\frac{1}{2\pi}\lim_{t\to 0} \frac{\int_{\c_s} \eta(M_{s'})} {\log \abs{t}} =
  \frac{1}{2\pi}\lim_{t\to 0} \frac{\int_{\c_s} \eta(L_{1,j} , L_{2,l} )}{\log \abs{t}} =
\begin{cases}
  \pm 1, & \text{if } s=s', \\
  0, & \text{otherwise},
\end{cases}
\end{equation*}
for $2\le i < k \le N$,
\begin{equation*}
\frac{1}{2\pi}\lim_{t\to 0} \frac{\int_{\c_s} \eta(M_{s'})} {\log \abs{t}} =
  \frac{1}{2\pi}\lim_{t\to 0} \frac{\int_{\c_s} \eta(L_{i,j},L_{k,l})} {\log \abs{t}} =
\begin{cases}
  \pm 1, & \text{if } s=s', \\
  0, & \text{otherwise},
\end{cases}
\end{equation*}
 and for $i=1, k>2$,
\begin{equation*}
\frac{1}{2\pi}\lim_{t\to 0} \frac{\int_{\c_s} \eta(M_{s'})} {\log \abs{t}} =
  \frac{1}{2\pi}\lim_{t\to 0} \frac{\int_{\c_{s}} -\eta(L_{1,j},L_{k,l})} {\log \abs{t}} =
\begin{cases}
  \pm 1, & \text{if } s=s', \\
  0, & \text{otherwise}.
\end{cases}
\end{equation*}
So the theorem is proved.
\end{proof}

\begin{remark}
For $ g \ge 1 $, in Theorem \ref{thm:limit} there are infinitely many different isomorphism classes of
curves over the algebraic closure of the base field in the family.

Namely, for $g=1$ the only two cases are
$ N=2 $ with both $ N_i $ equal 2,
and
$ N=3 $ and all $ N_i $ equal 1.
Using~\eqref{eqn:hyper1} and~\eqref{eqn:hyper2}
it is easy to see that the $j$-invariant is a non-constant function of $t$
in either case.

For $ g \ge 2 $, note that,
if a family of curves has a stable, singular fibre, and smooth general fibres
of genus $g$, it cannot  be isotrivial, since the image of the corresponding map to the moduli space
passes through a general point and a point on the boundary and hence cannot be constant.

In our family the fibre $ X_0 $ is the union of lines meeting transversely.
It is stable except if $N=2$ and one of the $N_i$ equals~2, the
other being larger than 2,
or if $N=3$ and precisely two of $N_1$, $N_2$ and $N_3$ equal~1.
In either case we can contract the $(-2)$-curves in $ X_0 $. The
resulting fibre now consists of two $\P^1$s
meeting transversely at $g+1$ points, which is stable. So if $g\geqslant 2$ the family is not isotrivial.
\end{remark}

Using our techniques, we can also obtain two independent elements
in $ \INTC $ in certain families of elliptic curves over a given real quadratic field.
The proof of the next proposition will show that
Lemma~\ref{lemma:basis}(1) applies to the family of curves over
$ \R $ obtained for each embedding of $ k $, and the regulators
are computed using the resulting basis of $ H_1(X,\Z)^- $,
with $ X $ the complex manifold associated to $ C \times_\Q \C $.

\begin{proposition}\label{prop:quadratic}
Let $k$ be a real quadratic field and $\OO_k$ its ring of integers.
\begin{enumerate}
\item [(1)]
For an integer $a$ with $\abs{a}>5$,
suppose $k$ is $\Q(\sqrt{a^2-16})$.
Consider the elliptic curve $ C $ over $ k $ defined by $ y^2 + (2x^2+ax+1)y + x^4 = 0 $.
If we write $4x^2+ax+1=(\v_1x+1)(\v_2x+1)$ in $k[x]$,
then the elements
\begin{equation*}
\widetilde{M}_l = 2\left\{\frac{y}{x^2}, \v_l x+1 \right\} \quad (l=1,2)
\end{equation*}
are in $\INTC$.
Their regulator $R=R(a)$ satisfies
\begin{equation*}
\lim_{\abs{a}\to\infty}\frac{R(a)}{\log^2\abs{a}} = 16
\,.
\end{equation*}

\item[(2)]
Fix $v \neq \pm 1$ in $\OO_k^*$ as well as $p$ and $q$ in $\OO_k$ with $pq=4$.
For $pv^n \ne \pm 2$, let $ C $ be the elliptic curve over $ k $
defined by $ y^2 + (2x^2+(p v^n + q v^{-n})x+1)y + x^4 = 0 $.
Then the elements
$ \widetilde{M}_1 = 2\left\{\frac{y}{x^2}, p v^n x+1 \right\} $
and
$ \widetilde{M}_2 = 2\left\{\frac{y}{x^2}, q v^{-n} x+1 \right\} $
are in $ \INTC $, and their regulator $ R = R(n) $ satisfies
\begin{equation*}
\lim_{n\to\infty}\frac{R(n)}{n^2} = 16 \log^2\abs{v}
\,.
\end{equation*}
\end{enumerate}
\end{proposition}

\begin{proof}
In~(1),
$\widetilde{M}_1$ and $\widetilde{M}_2$ are in $\INTC$ by the proof of \cite[Theorem~8.3]{Jeu06}.
Reading the transformation from~\eqref{eqn:hypercase1} to~\eqref{eqn:hyper1}
backwards, $C$ can be transformed into the curve defined by
$x(x+a)y(y+1)-1 = 0 $ and the elements become
$ 2\left\{\frac{y}{y+1}, \frac{x+\v_l}{x} \right\} $.
Replacing $x$ with $ax$, and letting $t=\frac{1}{a^2}$, we obtain
a curve defined by $x(x+1)y(y+1)-t = 0 $, with elements
$ 2\left\{\frac{y}{y+1}, \frac{x+\frac{\v_l}{a}}{x} \right\} $.
Clearly, $\lim_{\abs{a}\to\infty}\frac{\v_l}{a}$ equals~0 or~1,
and the two different embeddings of $ k $ into $ \C $ swap the
two cases. Using a limit argument directly on $ \int_{\c_t} \eta(\widetilde{M}_l) $
it is easy to see that we may, in the limit of $ R(a) $,
replace $ \v_l/a $ by its limit in $ \C $. The formula for $ R(a) $
then follows from Lemmas~\ref{lemma:basis} and~\ref{lemma:limit}.

For~(2), one first notices that the proof of~\cite[Theorem~8.3]{Jeu06}
also shows the $ \widetilde{M}_l $ are in $ \INTC $
because $ (p v^n x + 1 )( q v^{-n} x + 1) = 4 x^2 + (p v^n + q v^{-n}) x + 1 $.
The statement about the limit then follows as in~(1), using
$ a = p v^n + q v^{-n} $, $ \v_1 = p v^n $ and $ \v_2 = q v^{-n} $.
\end{proof}

\begin{remark}
In the first part of Proposition~\ref{prop:quadratic}, every real quadratic field with discriminant $ D $ occurs infinitely often
because $ c^2-Dd^2=1 $ has infinitely many integer solutions.
Also since $ C $ is already defined over $ \Q $, one could use the theory of
quadratic twists and the modularity of $C$ in order
to show the existence of two independent elements in $\INTC$ for $C/k$.
That would be far less explicit, but it would also give the
expected relation between the regulator and the $L$-function.

For the second part of Proposition~\ref{prop:quadratic},
using explicit calculations one checks that
the $j$-invariant is not a constant function of $n$,
for $ n \gg 0 $ is not an algebraic integer,
and, if the norm of $p$ in $\Q$ does not have absolute value $4$,
for $n \gg 0$ is not rational.
So for $ n \gg 0 $, the elliptic curve does not have complex
multiplication, and if the norm of $ p $ does not have absolute
value 4, the curve cannot be defined over~$ \Q $.
\end{remark}


\begin{thebibliography}{00}

\bibitem{SGA72}
{\em Groupes de monodromie en g\'eom\'etrie alg\'ebrique. {II}}.
\newblock Lecture Notes in Mathematics, Vol. 340. Springer-Verlag, Berlin,
  1973.
\newblock S{\'e}minaire de G{\'e}om{\'e}trie Alg{\'e}brique du Bois-Marie
  1967--1969 (SGA 7 II), Dirig{\'e} par P. Deligne et N. Katz.

\bibitem{Bass}
H.~Bass.
\newblock {\em Algebraic {$K$}-theory}.
\newblock W.A. Benjamin, Inc., New York-Amsterdam, 1968.

\bibitem{Be85}
A.~Beilinson.
\newblock Higher regulators and values of {$L$}-functions.
\newblock {\em J.Sov.Math.}, 30:2036--2070, 1985.

\bibitem{Bei80}
A.~A. Beilinson.
\newblock Higher regulators and values of ${L}$-functions of curves.
\newblock {\em Funct. Anal. Appl.}, 14:116--118, 1980.

\bibitem{Bernard}
D.~Bernard and M.~Patterson.
\newblock Computing with algebraic curves and {R}iemann surfaces: the
  algorithms of the {M}aple package {"algcurves"}.
\newblock In {\em Computational Approach to Riemann Surfaces}, volume 2013 of
  {\em Lecture Notes in Mathematics}, pages 67--123. Springer, Berlin, 2011.

\bibitem{Bl78}
S.~Bloch.
\newblock {\em Higher regulators, algebraic {$K$}-theory, and zeta functions of
  elliptic curves}, volume~11 of {\em CRM Monograph Series}.
\newblock American Mathematical Society, Providence, RI, USA, 2000.

\bibitem{BG}
S.~Bloch and D.~Grayson.
\newblock {$K_2$} and {$L$}-functions of elliptic curves: {C}omputer
  calculations.
\newblock In {\em Applications of Algebraic {$K$}-theory to algebraic geometry
  and number theory}, volume~55 of {\em Contemporary Mathematics}, pages
  79--88. American Mathematical Society, Providence, RI, USA, 1986.

\bibitem{Borel77}
A.~Borel.
\newblock Cohomologie de {${\rm SL}\sb{n}$} et valeurs de fonctions z\^eta aux
  points entiers.
\newblock {\em Ann. Scuola Norm. Sup. Pisa Cl. Sci. (4)}, 4(4):613--636, 1977.
\newblock Errata in vol. 7, p. 373 (1980).

\bibitem{dJ08}
R.~de~Jeu.
\newblock Further counterexamples to a conjecture of {B}eilinson.
\newblock {\em J. K-Theory}, 1(1):169--173, 2008.

\bibitem{De89}
C.~Deninger.
\newblock Higher regulators and {H}ecke {$L$}-series of imaginary quadratic
  fields {$\uppercase\expandafter{\romannumeral1}$}.
\newblock {\em Inventiones Mathematicae}, 96:1--69, 1986.

\bibitem{De90}
C.~Deninger.
\newblock Higher regulators and {H}ecke {$L$}-series of imaginary quadratic
  fields {$\uppercase\expandafter{\romannumeral2}$}.
\newblock {\em Annals of Mathematics}, 132:131--158, 1990.

\bibitem{Jeu06}
T.~Dokchitser, R.~de~Jeu, and D.~Zagier.
\newblock Numerical verification of {B}eilinson's conjecture for {$K_2$} of
  hyperelliptic curves.
\newblock {\em Compositio Mathematica}, 142:339--373, 2006.

\bibitem{Kimura}
K.~Kimura.
\newblock {$K_2$} of a Fermat quotient and the value of its {$L$}-function.
\newblock {\em K-Theory}, 10:73--82, 1996.

\bibitem{K-S}
U.~K{\"u}hn and S.~M{\"u}ller.
\newblock A geometric approach to constructing elements of ${K}_2$ of curves.
\newblock In preparation.

\bibitem{Noether}
M.~Noether.
\newblock Rationale ausf{\"{u}}hrungen der operationen in der theorie der
  algebraischen funktionen.
\newblock {\em Math. Ann.}, 23:311--358, 1883.

\bibitem{Otsubo}
N.~Otsubo.
\newblock On the regulator of fermat motives and generalized hypergeometric
  functions.
\newblock {\em J. reine angew. Math.}, 660:27--82, 2011.

\bibitem{rol}
K.~Rolshausen.
\newblock {\em \'El\'ements explicites dans $ K_2 $ d'une courbe elliptique}.
\newblock PhD thesis, Universit\'e de Strasbourg, 1996.

\bibitem{scha-rols98}
K.~Rolshausen and N.~Schappacher.
\newblock On the second ${K}$-group of an elliptic curve.
\newblock {\em J. Reine Angew. Math.}, 495:61--77, 1998.

\bibitem{Rosenberg}
J.~Rosenberg.
\newblock {\em Algebraic {$K$}-theory and its applications}, volume 147 of {\em
  GTM}.
\newblock Springer, Berlin, 1996.

\bibitem{Ross92}
R.~Ross.
\newblock {$K_2$} of elliptic curves with sufficient torsion over $\mathbb{Q}$.
\newblock {\em Compositio Mathematica}, 81:211--221, 1992.

\bibitem{Schneider}
P.~Schneider.
\newblock Introduction to the {B}eilinson {C}onjectures.
\newblock In {\em Beilinson's {C}onjectures on {S}pecial {V}alues of
  $L$-Functions}, volume~4 of {\em Perspect. Math.}, pages 1--35. Academic
  Press, Boston, MA, 1988.

\end{thebibliography}
\end{document}